\documentclass[11pt,leqno]{amsart}
\usepackage{hyperref}
\usepackage{mathrsfs}
\usepackage{tikz-cd}
\usepackage{amsmath,amsfonts,amssymb}
\usepackage{xfrac}
\usepackage{nicefrac}
\usepackage{color, xcolor}

\usepackage[margin=1.15in]{geometry}

\theoremstyle{plain}

\newtheorem{theorem}{Theorem}[section]
\newtheorem{lemma}[theorem]{Lemma}
\newtheorem{proposition}[theorem]{Proposition}
\newtheorem{corollary}[theorem]{Corollary}

\theoremstyle{definition}
\newtheorem{definition}[theorem]{Definition}

\newtheorem{example}[theorem]{Example}
\newtheorem{notation}[theorem]{Notation}

\numberwithin{equation}{section}

\DeclareMathOperator{\Aut}{Aut}

\DeclareMathOperator{\st}{St}
\DeclareMathOperator{\rist}{Rist}
\DeclareMathOperator{\C}{\mathcal{C}}

\renewcommand{\epsilon}{\varepsilon}

\title[GGS-groups acting on  trees of growing degrees]{GGS-groups acting on  trees of growing degrees}

\author[R. Skipper]{Rachel Skipper
} 
\address{Rachel Skipper: 
Department of Mathematics, University of Utah, 155 S 1400 E, Salt Lake City, Utah, 84112, USA}
\email{
rachel.skipper@utah.edu}

 \author[A. Thillaisundaram]{Anitha Thillaisundaram
 }
 
\address{Anitha Thillaisundaram: 
Centre for Mathematical Sciences, Lund University, 223 62 Lund, Sweden}
 \email{
 anitha.thillaisundaram@math.lu.se}



\begin{document}

\begin{abstract}
 We consider analogues of Grigorchuk-Gupta-Sidki (GGS-)groups acting on trees of growing degree; the so-called growing GGS-groups. These groups are not just infinite and do not possess the congruence subgroup property, but many of them are branch and have the $p$-congruence subgroup property, for a prime~$p$. Among them, we find  groups with  maximal subgroups only of finite index, and with infinitely many such maximal subgroups. These give the first examples of finitely generated branch groups with infinitely many finite-index maximal subgroups. Additionally, we prove that  congruence quotients of growing GGS-groups associated to a defining vector of zero sum give rise to Beauville groups. 

\end{abstract}

\maketitle


\section{Introduction}
Branch groups are groups acting spherically transitively on a spherically homogeneous infinite rooted tree and having normal subgroup structure similar to the subgroup structure of the full automorphism group of the tree. Early constructions,  which have many interesting properties, were produced by Grigorchuk~\cite{Grigorchuk} and Gupta and Sidki~\cite{Gupta}.  Their groups act on a special type of regular rooted tree: the $p$-regular rooted trees, for $p$ a prime. 

These pioneering examples were generalised to a family of GGS-groups, named after Grigorchuk, Gupta and Sidki. A GGS-group is generated by an automorphism $a$, which permutes the $p$ subtrees hanging from the root rigidly according to the permutation $(1\,2\,\cdots\,p)$, and a recursively-defined automorphism $b$, which fixes the first-level vertices of the tree and acts on their respective subtrees as $(a^{e_1},\ldots,a^{e_{p-1}},b)$, for some $e_1,\ldots,e_{p-1}\in \mathbb{F}_p$.

Certain branch groups acting on irregular trees of growing degree were studied by Fink in~\cite{Fink} and by Petschick in~\cite{Engel}. Fink used these groups to construct examples of branch groups with exponential word growth but no free subgroups and Petschick constructed the first examples of Engel branch $p$-groups that are not nilpotent; these latter branch groups act on trees of growing $p$-power degree. More generally, branch groups acting on non-regular rooted trees were considered by Segal~\cite{Segal} and Bondarenko~\cite{Bondarenko}.

In this paper, we introduce a generalisation of the GGS-groups to trees of growing degrees, where the vertices at level~${i-1}$ of the tree have $p^{m_i}$ children, for $m_i\in \mathbb{N}$ with $m_1< m_2<\cdots$. We refer to these analogous GGS-groups as \emph{growing GGS-groups}. These groups do not share all properties of GGS-groups, for instance the analogue of the automorphism $b$ is not of finite order. We refer the reader to Section~\ref{sec:properties} for details. 
The growing GGS-groups also do not possess the congruence subgroup property, however a large subfamily of them does possess the $p$-congruence subgroup property and the weak congruence subgroup property; see Subsections~\ref{sec:subgroups-of-Aut-T} and~\ref{sec:CSP} respectively for the definitions and these results. This subfamily provides the first examples of finitely generated branch groups with such a combination of congruence subgroup properties. In particular this gives the first examples of finitely generated branch groups shown to have the $p$-congruence subgroup property but not the congruence subgroup property; cf. \cite{pcongruence, PR, DNT}.

Our main result concerns the maximal subgroups of growing GGS-groups. 
Now, the study of  maximal subgroups of branch groups acting on regular rooted trees, in particular, $p$-regular rooted trees, is well established, starting from the work of Pervova~\cite{Pervova3, Pervova2,  Pervova4}, who proved that the torsion Grigorchuk groups and
torsion GGS-groups do not contain maximal subgroups of infinite index.
In~\cite{KT},  Pervova's result was extended to torsion multi-EGS groups, also called torsion generalised multi-edge spinal groups. Bondarenko~\cite{Bondarenko}  gave
 the first, though non-explicit, examples of  finitely
generated branch groups that do have maximal subgroups of infinite
index.  His method does not apply to
groups acting on the  2-regular and 3-regular rooted trees. However, recently Francoeur and Garrido~\cite{FG} provided the first examples of finitely generated branch groups, acting on the $2$-regular rooted tree,  with maximal subgroups of infinite index. Their examples are the non-torsion {\v{S}}uni\'{k} groups acting on the 2-regular rooted tree. See \cite{Francoeur, Francoeur-paper,FT} for further related results.

In this paper, we extend Pervova's result to  branch growing GGS-groups that satisfy a natural torsion-type condition, by showing that they do not have maximal subgroups of infinite index, providing the first such examples among finitely generated branch groups acting on irregular rooted trees.

Now for a growing GGS-group  associated to a defining vector $\mathbf{e}=(e_1,\ldots,e_{p^{m_1}-1})$, the torsion-type condition is the following; compare with~\cite{vov}:
\begin{equation}\label{eq:periodicity}
\sum_{k=1}^{p^{m_1-i}-1} e_{kp^i} \equiv 0 \pmod {p^{i+1}}\quad\text{for all  }i\in\{0,\ldots,m_1-1\}.
\end{equation}
Note that although none of the groups considered here are torsion, we call this the torsion-type condition as it is the necessary and sufficient condition for a GGS-group acting on the $p^n$-regular rooted tree, for fixed $n$, to be torsion.

\begin{theorem} \label{start} Let $G$ be a  branch growing GGS-group satisfying condition~\eqref{eq:periodicity}. Then $G$ has countably many maximal subgroups, all of which have finite index. In particular, for every prime number~$q$, the group $G$ has a normal maximal subgroup of index~$q$.
\end{theorem}

This gives the first examples of finitely generated branch groups with infinitely many (normal) maximal subgroups only of finite index. Previous examples, such as the branch GGS-groups acting on the $p$-regular rooted tree, the torsion multi-EGS groups and the torsion Grigorchuk groups all possess only finitely many finite-index maximal subgroups. Even the non-torsion {\v{S}}uni\'{k} groups, which possess maximal subgroups of infinite index, have only finitely many maximal subgroups of finite index.

For $G$ a branch growing GGS-group, we further have that $G$ has uncountably many $\Aut T$-equivalence classes of weakly maximal subgroups, which are all distinct from classes of parabolic subgroups; see Subsection~\ref{sec:weakly-maximal}. We recall that a \emph{weakly maximal subgroup} is a subgroup that is maximal with respect to being of infinite index.

In the last section, we find quotients of growing GGS-groups that are Beauville groups. Recall that Beauville groups are defined in terms of Beauville surfaces: a \emph{Beauville surface} is a compact complex surface isomorphic to $(\C_1\times \C_2)/H$, where  $\C_1$ and $\C_2$ are algebraic curves of genus at least 2, and $H$ is a finite group acting freely on $\C_1\times \C_2$ by holomorphic transformations, with the group $H$ acting faithfully on the curves $\C_i$ such that $\C_i/H\cong \mathbb{P}_1(\mathbb{C})$ and the covering map $\C_i\rightarrow \C_i/H$ is ramified over three points, for $i\in\{1,2\}$.
The group~$H$ is then said to be a \emph{Beauville group}. For more background on Beauville surfaces, we refer the reader to~\cite{BCG}.

Beauville groups have recently gained much attention; see the surveys \cite{fai, fai2, jon}. The abelian Beauville groups were classified by Catanese~\cite{Catanese}: a finite abelian group $A$ is a Beauville group if and only if $A\cong C_n \times C_n$ for $n> 1$ with $\gcd(n, 6) = 1$. After abelian groups, the most natural finite groups to consider are nilpotent groups, and the determination of nilpotent Beauville groups can be reduced to the case of $p$-groups.

There has been considerable effort in finding infinite families of Beauville $p$-groups; see for instance \cite{BBF, BBPV, FAG, Gul, SV}. Recently, G\"{u}l and Uria-Albizuri~\cite{GUA}  showed that the quotients of torsion GGS-groups acting on the $p$-regular rooted tree admit Beauville  structures. This has been generalised to torsion infinite GGS-groups acting on the $p^n$-regular rooted tree, including the case $p=2$, by Di Domenico, G\"{u}l and Thillaisundaram~\cite{DGT}.

One finds quotients of growing GGS-groups by their level stabilisers that are Beauville, but this is not always the case. In particular, we establish the following.

\begin{proposition}\label{Beauville}
 Let $G$ be a growing GGS-group with defining vector~$\mathbf{e}$ which is not congruent to the zero vector modulo~$p$. If
 \begin{enumerate}
     \item[(i)] $\sum_{i=1}^{p^{m_1}-1} e_i =0$, then $G/\st_G(n)$ is a Beauville group for every $n\ge 3$;
     \item[(ii)] 
  $\sum_{i=1}^{p^{m_1}-1} e_i \not\equiv 0 \pmod p$, then $G/\st_G(n)$ is not a Beauville group for every $n$.
  \end{enumerate}
\end{proposition}

\medskip

\noindent\emph{Organisation.} In Section~\ref{sec:2} we recall the preliminaries for groups acting on rooted trees of growing degree. In Section~\ref{sec:properties} we define the growing GGS-groups and establish their first properties, such as their abelianisation, branching properties, and their congruence subgroup properties. In Section~\ref{sec:maximal}, we prove our main result regarding maximal subgroups, and finally in Section~\ref{sec:Beauville} we prove Proposition~\ref{Beauville}.

\medskip

\noindent\textit{Notation.}
Throughout, we  use left-normed commutators, for example, $[x,y,z] = [[x,y],z]$.

\medskip

\noindent\textbf{Acknowledgements.} We are grateful to Khalid Bou-Rabee for his discussions with the first author a few years ago, which gave rise to the idea of studying branch groups acting on trees of growing degrees. We also thank Jorge Fari\~{n}a-Asategui for  useful discussions and the referee for the helpful comments.

This research was supported by a London Mathematical Society Grant Scheme 2 (Visits to the UK) and by the University of Lincoln's Institute of Advanced Studies. The first author was supported by the GIF, grant I-198-304.1-2015, ``Geometric exponents of random walks and intermediate growth groups", NSF grant no. DMS-2005297 ``Group Actions on Trees and Boundaries of Trees", and the European Research Council (ERC) under the European Union’s Horizon 2020 research and innovation program (grant agreement No.725773).




\section{Preliminaries} \label{sec:2} In the present section we recall
the notion of branch groups and establish prerequisites for the rest
of the paper.  For more details,
see~\cite{BarthGrigSunik,NewHorizons}.

\subsection{The trees of growing degrees and their automorphisms}
All trees in this paper will be rooted. 
A spherically homogeneous tree is a tree where, for any given $n$, all vertices of distance~$n$ from the root have the same degree.
Let $p$ be an odd prime and let $p^{\overline{m}}=(p^{m_i})_{i=1}^\infty$, where $m_1<m_2<m_3<\cdots$ is a sequence of increasing positive integers. Define $T=T_{\overline{m}}$ to be the spherically homogeneous tree with branching sequence $p^{\overline{m}}=(p^{m_i})_{i=1}^\infty$. In other words $T$ will be the tree constructed as follows: for each $i\in\mathbb{N}$, let $X_i$ be the alphabet $\{1,2,\ldots,p^{m_i}\}$ and consider the set of finite sequences $X^*= \{\varnothing\}\cup\bigcup_{n\in\mathbb{N}}\big(\prod_{i=1}^n X_i )$. If $u$ is a vertex in $\prod_{i=1}^n X_i$ then we place an edge from $u$ to $v$ whenever $v=ux_{n+1}$ for some $x_{n+1}\in X_{n+1}$. In particular, the root of~$T$ corresponds to the empty word~$\varnothing$. There is a natural length function
on $X^*$ and the words $w$ of length $|w| = n$, representing vertices that are at distance $n$ from the root, make up the \textit{$n$th layer} of the tree. Hence the set~$X^*$ can be identified with the vertices of the tree~$T$, and we will use this identification freely. The \textit{boundary} $\partial T$, consisting of all infinite simple rooted paths, is in one-to-one correspondence with the words in the infinite product $\prod_{i=1}^\infty X_i$.  

All concepts associated to~$T$ that are defined here in Section~2 hold more generally for an arbitrary spherically homogeneous tree, but for notational convenience, we will only define them for~$T$.

We write $T_u$ for the full rooted subtree of $T$ that has its root at
the vertex~$u$ and includes all vertices $v$ with $u$ a prefix of~$v$. As $T$ is spherically homogeneous, for any two vertices $u$ and $v$  of a given level, the
subtrees $T_u$ and $T_v$ are isomorphic under the map that deletes the
prefix~$u$ and replaces it by the prefix~$v$. We  write
$T_n$ to denote the subtree rooted at a generic vertex of level~$n$. Defining $\sigma^n(\overline{m})=(m_i)_{i=n+1}^\infty$ for each $n$, we see that $T_n\cong T_{\sigma^{n}(\overline{m})}$.

We observe that every automorphism of the rooted tree $T$ fixes the root, and that the
orbits of $\Aut T$ on the vertices of the tree~$T$ are precisely its
layers.  Consider an automorphism $f \in \Aut T$ and denote the image of a
vertex~$u$ under $f$ by~$u^f$.  For a vertex~$u$ at level~$n$ and a letter $x \in X_{n+1}$ we have $(ux)^f=u^fx'$
where $x' \in X_{n+1}$ is uniquely determined by $u$ and~$f$.  This induces
a permutation $f(u)$ of $X_{n+1}$ so that $(ux)^f = u^f x^{f(u)}$. We refer to $f(u)$ as the \emph{labelling} of $f$ at $u$.
The automorphism~$f$ is called \textit{rooted} if $f(u)=1$ for $u\ne
\varnothing$.  It is called \textit{directed}, with directed path~$\ell$ for some
$\ell \in \partial T$, if the support $\{u \mid f(u)\ne1 \}$ of its
labelling is infinite and contains only vertices at distance~$1$ from~$\ell$.
The \textit{section} of~$f$ at a vertex~$u$ at level~$n$ is the unique automorphism~$f_u$ of~$T_n$ given by the condition $(uv)^f = u^f
v^{f_u}$ for $v \in \prod_{i=n+1}^{\infty} X_i$.

\subsection{Subgroups of the automorphism group of the tree}\label{sec:subgroups-of-Aut-T}
Let $G$ be a subgroup of $\Aut T$. The \textit{vertex stabiliser} $\text{st}_G(u)$ is the subgroup  of
elements in~$G$ that fix the vertex~$u$.  For $n \in \mathbb{N}$, the
\textit{$n$th level stabiliser} $\st_G(n)= \cap_{|v|=n} \text{st}_G(v)$
is the subgroup  of automorphisms that fix all vertices at
level~$n$.  Note that elements in $\st_G(n)$ necessarily fix all vertices up to
level~$n$ and that $\st_G(n)$ has finite index and is normal in~$G$.

The full automorphism group $\Aut T$ is a profinite group, and we have 
\[
\Aut T= \varprojlim_{n\in\mathbb{N}} \Aut T_{[n]},
\]
where $T_{[n]}$ denotes the subtree of~$T$ on the finitely many
vertices up to level~$n$. The topology of $\Aut T$ is defined by the
open subgroups $\st_{\Aut T}(n)$ for $n \in \mathbb{N}$.  Now for $G\le \Aut T$, the level
stabilisers $\st_G(n)$ for $n \in \mathbb{N}$, form a natural family of
so-called \emph{principal congruence subgroups} for~$G$,  and we say the subgroup $G\le\Aut T$
has the \textit{congruence subgroup property} if the profinite
topology and the congruence topology on~$G$ coincide. Equivalently, if for
every subgroup $H$ of finite index in $G$, there exists some $n$ such
that $\st_G(n)\subseteq H$.
A weaker version of the congruence subgroup property is the $p$-congruence subgroup property for a fixed prime $p$:  a subgroup~$G\le \Aut T$ has the  $p$-\emph{congruence subgroup property} if for every normal subgroup~$H$ of finite $p$-power index in~$G$, there
exists some $n$ such that $\mathrm{St}_G(n)\subseteq H$; see~\cite{pcongruence}.  Another weaker version of the congruence subgroup property is the \emph{weak congruence subgroup property} introduced by Segal,  where every finite-index subgroup contains $\text{St}_G(n)'$, for some $n$; see~\cite{Segal}.

For $n\in\mathbb{N}$, every $g\in \st_{\Aut T} (n)$ can be identified with a sequence
$g_1,\ldots,g_{p^{m_1+\cdots +m_{n}}}$ of elements of $\Aut T_n$, where $p^{m_1+\cdots +m_{n}}$ is the number of vertices at level~$n$.  Denoting these vertices by $u_1, \ldots,
u_{p^{m_1+\cdots +m_{n}}}$, there is a natural
isomorphism
\[
\st_{\Aut T}(n) \cong \prod\nolimits_{i=1}^{p^{m_1+\cdots +m_{n}}} \Aut T_{u_i}
\cong \Aut T_n \times \overset{p^{m_1+\cdots +m_{n}}}{\dots} \times \Aut T_n.
\] 
This decomposition of~
$g$ into its sections $(g_1,\ldots,g_{p^{m_1+\cdots +m_n}})$ defines the isomorphism
\[
\psi_n \colon \st_{\Aut T}(n) \longrightarrow \Aut T_n \times \overset{p^{m_1+\cdots +m_{n}}}{\dots} \times \Aut T_n.
\]    

Let  $v\in \prod_{i=1}^{n} X_i$ be a vertex of length~$n$. We further define 
\[
\varphi_v :\text{st}_{\Aut T}(v) \longrightarrow \Aut T_{v} \cong\Aut T_n
\]
to be the restriction of $\text{st}_{\Aut T}(v)$ to $T_v$.


Next, the subgroup $\text{rist}_G(u)$, consisting of all automorphisms in~$G$ that fix all vertices~$v$ of~$T$ not having $u$ as a prefix, is
called the \textit{rigid vertex stabiliser} of~$u$ in~$G$. 
For $n\in\mathbb{N}$, the
\textit{rigid $n$th level stabiliser} is the product
\[
\rist_G(n)=\prod\nolimits_{i=1}^{p^{m_1+\cdots+m_{n}}} \text{rist}_G(u_i) \trianglelefteq G
\]
of the rigid vertex stabilisers of the vertices $u_1, \ldots, u_{p^{m_1+\cdots+m_{n}}}$
at level~$n$. 

Recall that a subgroup $G\le\Aut T$ is said to act \textit{spherically transitively} if it acts transitively on every layer of~$T$.

\begin{definition}
A group~$G$ acting on $T=T_{\overline{m}}$ is a  \emph{branch group} if $G$ acts spherically transitively  and all rigid level stabilisers
$\rist_G(n)$ are of finite index in~$G$. A group $G$ acting on $T=T_{\overline{m}}$ is a  \emph{weakly branch group} if $G$ acts spherically transitively and all rigid level stabilisers
$\rist_G(n)$ are non-trivial. 
\end{definition}


\section{The growing GGS-groups and first properties}\label{sec:properties}
We now define a growing GGS-group $G$ acting on a tree $T=T_{\overline{m}}$ with branching $p$-power exponents given by $\overline{m}=(m_1,m_2,m_3,\ldots)$. For $i\in \mathbb{N}$, we will write $a_{i-1}$ to represent the rooted automorphism corresponding to the cycle $(1\,2\,\cdots\,p^{m_i})$.  Let $e_1,\ldots,e_{p^{m_1}-1}\in \{0,\pm1,\ldots,\pm (p^{m_2}-1)\}$ with not all $e_i$ being zero.  Then the \emph{growing GGS-group}~$G=G_{\mathbf{e}}$ acting on the rooted tree $T$ and defined by the vector $\mathbf{e}=(e_1,\ldots,e_{p^{m_1}-1})$, is the group generated by the rooted automorphism~$a_0$ and the directed automorphism~$b_0$ defined as follows:
\[
\psi_1(b_0)=(a_1^{\,e_1},a_1^{\,e_2},\ldots, a_1^{\,e_{p^{m_1}-1}},b_1)
\]
where recursively for $n\ge 2$,
\[
\psi_1(b_{n-1})=(a_n^{\,e_1},a_n^{\,e_2},\ldots, a_n^{\,e_{p^{m_1}-1}}, 1,\overset{p^{ m_n}-p^{m_1}}\ldots,1,b_n).
\]

Write $G_n$ for the $n$th shifted group, that is, the growing GGS-group generated by $a_{n}$ and~$b_n$ which acts on the tree with branching indices $(p^{m_{n+1}},p^{m_{n+2}},\ldots)$.  
Note that $a_n, b_n$, and $G_n$ depend on the choice of $\overline{m}$. Since we will work with arbitrary increasing sequences, to simplify notation we often omit $\overline{m}$ and write $G=G_0$, $a=a_0$ and $b=b_0$. Recalling that $\sigma$ denotes the shift map of a sequence, i.e. $\sigma((x_i)_{i=1}^\infty)=(x_i)_{i=2}^\infty$, we define the following concepts:

\begin{definition}
We say that a growing GGS-group~$G$
is \textit{$\sigma$-fractal} if, for each $n$ and vertex~$v$ at the $n$th level, the group~$\varphi_v(\text{st}_G(v))$ coincides
with~$G_n$.
Furthermore we say that the group~$G$ is \emph{strongly $\sigma$-fractal} if $\varphi_x(\text{St}_G(1))=G_1$ for every vertex $x$ at  the first level, and we say that 
the group~$G$ is \emph{super strongly $\sigma$-fractal} if, for each $n$, we have $\varphi_v(\text{St}_G(n))=G_n$ for every vertex~$v$ at the $n$th level; compare \cite[Def.~2.4]{Jone}. 
\end{definition}

We note that traditionally, the terms fractal, strongly fractal and super strongly fractal have been used only in the context of self-similar groups and  $d$-regular rooted trees, for $d\ge 2$, where a group $H \leq \Aut T$ is said to be \emph{self-similar} if for every $h\in H$ and every vertex $u$, the section~$h_u$, of $h$ at $u$, is an element of $H$.
However, we naturally extend the definitions to spherically homogeneous trees and non-self-similar actions. 

\subsection{Abelianisation of growing GGS-groups}  
Let $G$ be a growing GGS-group. To determine the abelianisation of $G$, we first consider an infinite collection of free products
\[
H_i=\langle \hat{a}_i , \hat{b}_i\mid \hat{a}_i^{p^{m_{i+1}}}=1\rangle =\langle \hat{a}_i\rangle *\langle \hat{b}_i\rangle \cong \nicefrac{\mathbb{Z}}{{p^{m_{i+1}}}\mathbb{Z}}*\mathbb{Z},
\]
for $i\in \mathbb{N}\cup\{0\}$. For each such $i$, there is a unique epimorphism $\pi_i:H_i\rightarrow G_i$,  which sends $\hat{a}_i\mapsto a_i$ and $\hat{b}_i\mapsto b_i$. The map $\pi_i$ induces an epimorphism from $\nicefrac{H_i}{H_i'}\cong \langle \hat{a}_i\rangle \times \langle \hat{b}_i\rangle\cong  \nicefrac{\mathbb{Z}}{{p^{m_{i+1}}}\mathbb{Z}}\times \mathbb{Z}$ onto $\nicefrac{G_i}{G_i'}$, which we will now show is an isomorphism. To do so, we proceed analogously to~\cite[Sec.~4.1]{AKT}.

We consider $h\in H_i$, which can be uniquely represented in the form
\begin{equation*}
h=\hat{a}_i^{s_1}\cdot \hat{b}_i^{\beta_1}\cdot \hat{a}_i^{s_2}\cdot \,\cdots\, \cdot \hat{a}_i^{s_{\ell}} \cdot\hat{b}_i^{\beta_{\ell}}\cdot \hat{a}_i^{s_{\ell+1}}
\end{equation*}
where $\ell\in\mathbb{N}\cup\{0\}$ and $s_1,\ldots,s_{\ell+1}\in \nicefrac{\mathbb{Z}}{{p^{m_{i+1}}}\mathbb{Z}}$, $\beta_1,\ldots,\beta_{\ell}\in \mathbb{Z}$ with
\[
s_j\not\equiv 0\quad\text{(mod $p^{m_{i+1}}$)}\qquad \text{for }j\in\{2,\ldots,\ell\},
\]
and $\beta_k\ne 0$ for each $k\in\{1,\ldots,\ell\}$.

We denote by $|h|=\ell$ the \emph{length} of $h$, with respect to the factor~$\hat{b}_i$. For $h_1,h_2\in H_i$ we have
\begin{equation}\label{eq:product-h1-h2}
|h_1h_2|\le |h_1|+|h_2|.
\end{equation}
Also for $h\in H_i$, we define the \textit{exponent maps}
\[
  \begin{split}
    \epsilon_{\hat a_i}(h) & =
    \sum\nolimits_{j=1}^{\ell+1} s_j \in \nicefrac{\mathbb{Z}}{{p^{m_{i+1}}}\mathbb{Z}} \quad \text{and} \quad  \epsilon_{\hat{b}_i}(h) = \sum\nolimits_{k=1}^{\ell} \beta_{k} \in
    \mathbb{Z} 
  \end{split}
\]
with respect to the generating set $\{\hat{a}_i,
\hat{b}_i\}$.

The epimorphism
\begin{equation*} 
H_i \rightarrow
  \nicefrac{\mathbb{Z}}{{p^{m_{i+1}}}\mathbb{Z}} \times \mathbb{Z}, \quad h
  \mapsto (\epsilon_{\hat a_i}(h), \epsilon_{\hat b_i}(h))
\end{equation*}
has kernel $H_i'$. 
Further the group $L(H_i) := \langle
\hat{b}_i\rangle^{H_i}$ is the kernel of the epimorphism
\[
H_i \rightarrow \nicefrac{\mathbb{Z}}{{p^{m_{i+1}}}\mathbb{Z}}, \quad h \mapsto \epsilon_{\hat a_i}(h).
\]
Here for a group~$\Gamma$ and a subset $Y \leq \Gamma$ we denote by $\langle Y \rangle^\Gamma$ the normal closure of~$Y$ in~$\Gamma$. Each element $h \in L(H_i)$ can be uniquely represented by a word of the
form
\begin{equation*} 
h= (\hat{c}_1)^{\hat{a}_i^{t_1}} \cdots
  (\hat{c}_{\ell})^{\hat{a}_i^{t_{\ell}}},
\end{equation*}
where $\ell\in \mathbb{N}\cup\{0\}$ and $t_1, \ldots, t_{\ell} \in
\nicefrac{\mathbb{Z}}{{p^{m_{i+1}}}\mathbb{Z}}$ with $t_j\not \equiv t_{j+1} \pmod{p^{m_{i+1}}}$ for
$j\in \{1,\ldots,\ell-1\}$, and for each $j\in \{1,\ldots,\ell\}$, 
\begin{equation*} 
\hat{c}_j  \in
\langle \hat{b}_i \rangle \backslash \{1\}.
\end{equation*}

Let $\alpha_{i+1}$ denote the cyclic permutation of the factors of $H_{i+1}\times
\overset{p^{m_{i+1}}}{\dots} \times H_{i+1}$ corresponding to the $p^{m_{i+1}}$-cycle $(1 \, 2
\, \cdots \, p^{m_{i+1}})$. Consider the homomorphism
\[
\Phi_i \colon L(H_i) \longrightarrow H_{i+1}\times \overset{p^{m_{i+1}}}{\dots} \times H_{i+1} \quad\text{for }i\in\mathbb{N}\cup\{0\}
\]
given by 
\[
\Phi_i(\hat{b}_i^{\,\hat{a}_i^k}) =
(\hat{a}_{i+1}^{\,e_1},\ldots,\hat{a}_{i+1}^{\,e_{p^{m_1}-1}},1,\overset{p^{m_{i+1}}-p^{m_1}}\ldots,1,\hat{b}_{i+1})^{\alpha_{i+1}^{\,k}}
\quad \text{for  $k \in
  \nicefrac{\mathbb{Z}}{{p^{m_{i+1}}}\mathbb{Z}}$.}
\]

The following result is proved exactly as in~\cite[Lem.~4.1]{AKT}.

\begin{lemma} \label{ceil} Let $H_i$ be as above, and $h \in L(H_i)$ with
  $\Phi_i(h)=(h_1,h_2,\ldots,h_{p^{m_{i+1}}})$.  Then $\sum_{j=1}^{p^{m_{i+1}}} |h_j|
  \le |h|$, and $|h_j| \le \lceil
  \frac{|h|}{2}\rceil$ for each $j \in \{1,\ldots,p^{m_{i+1}}\}$.
\end{lemma}

\begin{lemma}
Let $h\in L(H_i)$ and suppose $\Phi_i(h)=(h_1, h_2, \dots, h_{p^{m_{i+1}}})$. Then $\epsilon_{\hat{b}_i}(h)=\sum_{j=1}^{p^{m_{i+1}}} \epsilon_{\hat{b}_{i+1}}(h_j)$.
\end{lemma}
\begin{proof}
This follows from the definition of $\Phi_i$ and the fact that $H_i$ is a free product.
\end{proof}

Consider the
  subgroup $K^{(0)}=\bigcup_{n=0}^{\infty} K_n^{(0)}$ of $H_0$, for  $ K_0^{(0)}=\{1\} $ and
  \[
K_n^{(0)} = \Phi_0^{-1}(K_{n-1}^{(1)} \times \overset{p^{m_1}}\dots \times K_{n-1}^{(1)})  \quad \text{ for $n\in\mathbb{N}$,}
  \]
  where $K_{n-1}^{(1)}$ is recursively defined for the group~$H_1$, and 
  \[
  K_n^{(1)} = \Phi_1^{-1}(K_{n-1}^{(2)} \times \overset{p^{m_2}}\dots \times K_{n-1}^{(2)})
  \]
  with $K_{n-1}^{(2)}$ recursively defined for the group~$H_2$, etc. Here $K_0^{(n)}=1$ for all $n\in\mathbb{N}$.
The following proposition provides a recursive presentation for a growing GGS-group.  It was proved more generally by
Rozhkov~\cite{Rozhkov}.

\begin{proposition} \label{K} Let $G$
  be a growing GGS-group, and set $H=H_0$ and $K=K^{(0)}$, as defined above. Then $K\le L(H) = \langle \hat{b}_0 \rangle
  ^H$, and $K$ is normal in~$H$.  Moreover, the epimorphism $\pi\colon
  H \rightarrow G$ given by $\hat{a}_0 \mapsto a_0$, $\hat{b}_0 \mapsto
  b_0$, has $\ker(\pi)=K$.  In particular
  $G \cong \nicefrac{H}{K}$.
\end{proposition}

Analogous to~\cite[Prop.~4.3]{AKT}, we obtain the abelianisation of a growing GGS-group.

\begin{proposition}
  \label{abelianization}
Let~$G$ be a growing GGS-group. Then the map $H \rightarrow
  \nicefrac{\mathbb{Z}}{{p^{m_1}}\mathbb{Z}} \times \mathbb{Z}$ factors through $\nicefrac{G}{G'}$.  Consequently,
  \[
  \nicefrac{G}{G'} \cong \nicefrac{H}{H'} \cong \nicefrac{\mathbb{Z}}{{p^{m_1}}\mathbb{Z}}\times \mathbb{Z}.
  \]
\end{proposition}

\medskip

As above, let $G=\langle a,b\rangle$ be a growing GGS-group, and let $\pi \colon H \rightarrow G$ be the natural epimorphism
with $H$ as above.  The \textit{length} of $g \in G$ is
\[
|g| = \min \{ |h| \mid h \in \pi^{-1}(g) \}.
\]
With reference to Equation~\eqref{eq:product-h1-h2}, we have that for $g_1,g_2
\in G$,
\[
  |g_1 g_2| \leq |g_1| + |g_2|.
\]
Also, using Proposition~\ref{abelianization} we define
$\epsilon_a(g)\in \nicefrac{\mathbb{Z}}{p^{m_1}\mathbb{Z}}$ and $\epsilon_{b}(g) \in
\mathbb{Z}$ via any pre-image $h \in \pi^{-1}(g)$:
\[
  \epsilon_a(g) =
  \epsilon_{\hat a}(h) \quad\text{and}\quad \epsilon_{b}(g)=\epsilon_{\hat b}(h).
\]

The following is a direct consequence of Lemma~\ref{ceil}.

\begin{lemma} \label{shortening} 
Let $G$ be a growing GGS-group and let $g \in \st_G(1)$ with  $\psi_1(g)=(g_1,\ldots,g_{p^{m_1}})$.
  Then $\sum_{j=1}^{p^{m_1}} |g_j| \le |g|$, and
  $|g_j| \le \lceil \frac{|g|}{2}\rceil$ for each
  $j\in \{1,\ldots,p^{m_1}\}$.

  In particular, if $|g|>1$ then $|g_j| < |g|$
  for every $j \in \{1,\ldots,p^{m_1}\}$.
\end{lemma}

\subsection{Branching properties}
In this subsection, we prove a series of lemmas that allow us to determine when certain growing GGS-groups are branch (cf. Theorem~\ref{lem:branch-gamma-3}). We begin with the following two simple but useful results.
As in the case of the GGS-groups acting on the $p$-regular rooted tree, for~$G$ a growing GGS-group acting on $T=T_{\overline{m}}$, we have $\st_G(1)=\langle b\rangle^G$.

\begin{lemma}\label{lem:first-over-derived}
Let $G$ be a growing GGS-group. 
Then $\st_G(1)/G'\cong \mathbb{Z}$, $\st_G(1)/\st_G(1)'\cong \mathbb{Z}^{p^{m_1}}$ and $G'/\gamma_3(G)\cong  
C_{p^{m_1}}$.
\end{lemma}

\begin{proof}
As mentioned  above, we have \[
\st_G(1)=\langle b,b^a,\ldots,b^{a^{p^{m_1}-1}}\rangle.
\]
As $b^{a^i}\equiv b^{a^j} \pmod {G'}$ for $i,j\in \mathbb{Z}/p^{m_1}\mathbb{Z}$, it then follows from Proposition~\ref{abelianization} that
$\st_G(1)/G'=\langle b\rangle G'/G' \cong \mathbb{Z}$,
as required.

Also, since $\psi_1(\st_G(1)')\le G_1'\times \overset{p^{m_1}}\dots\times G_1'$ we have  $\st_G(1)/\st_G(1)'\cong \mathbb{Z}^{p^{m_1}}$.

For the third statement, we observe that $G'/\gamma_3(G)=\langle [a,b]\rangle\gamma_3(G)/\gamma_3(G)$. Since $$
1=[a^{p^{m_1}},b]\equiv [a,b]^{p^{m_1}} \pmod {\gamma_3(G)},
$$
it follows that $|G':\gamma_3(G)|\mid p^{m_1}$. To show equality, it suffices to show that $[a^{p^{m_1-1}},b]\notin \gamma_3(G)$. First we note that $[b,b^{a^i}]=[b,b[b,a^i]]\equiv [b,b] \pmod {\gamma_3(G)}$ for any $i\in \mathbb{Z}/p^{m_1}\mathbb{Z}$, and hence $\st_G(1)'\le \gamma_3(G)$. 
We have 
\begin{align*}
    \psi_1([a^{p^{m_1-1}},b])&=\psi_1\big(\big((b^{-1})^{a^{p^{m_1-1}}}b\big)\big)\\
    &=(a_1^{\,*},\overset{p^{m_1-1}-1}\ldots, a_1^{\,*}, b_1^{\,-1}a_1^{\,e_{p^{m_1-1}}},a_1^{\,*},\ldots, a_1^{\,*}, a_1^{-e_{p^{m_1}-p^{m_1-1}}} b_1),
\end{align*}
where $*$ represents unspecified exponents. Let $\mathbf{v}=(0,\overset{p^{m_1-1}-1}\ldots,0,-1,0,\ldots,0,1)$ denote the vector encoding the total exponents of~$b_1$ in each of the components of $\psi_1\big(\big((b^{-1})^{a^{p^{m_1-1}}}b\big)\big)$. Writing $\mathbf{V}$ for the span of such $b_1$-exponent vectors of elements in~$\gamma_3(G)$, we claim that $\mathbf{v}\notin \mathbf{V}$. To see this, it suffices to consider the $b_1$-exponent vectors of the elements $[a,b,a]$, $[a,b,a]^a$, $\ldots$, $[a,b,a]^{a^{p^{m_1}-1}}$, since the images of these elements  generate $\gamma_3(G)/\st_G(1)'$. We see that 
\[
\mathbf{V}=\langle (-1,2,-1,0,\ldots,0), (0,-1,2,-1,0,\ldots,0),\,\ldots\,,(2,-1,0,\ldots,0,-1)\rangle,
\]
and as a direct check shows that $(0,\overset{p^{m_1-1}-1}\ldots,0,-1,0,\ldots,0,1)\notin\mathbf{V}$, the result follows.
\end{proof}

\begin{notation}
Following the notation in~\cite{DFG}, let $\mathcal{F}$ 
be the set of defining vectors that are non-zero modulo~$p$. 
\end{notation}

\begin{lemma}\label{lem:fractal}
Let~$G$ be a growing GGS-group with defining vector $\mathbf{e}\in \mathcal{F}$. Then the group~$G$ acts spherically transitively on~$T$ and~$G$ is both $\sigma$-fractal and strongly $\sigma$-fractal.
\end{lemma}

\begin{proof}
For the first statement, since $\overline{m}$ is arbitrary and $\mathbf{e}\in \mathcal{F}$, it suffices to show that~$G$ acts transitively at the first level and that $\varphi_v(\text{st}_G(v))=G_1$ for $v\in X_1$. Clearly the group~$G$ acts transitively at the first level since $a$ does. The rest of the statement follows from the fact that for some $i$ we have
\[
\varphi_i(b)=a_1^{e_i}\quad\text{and}\quad\varphi_i(b^{a^i})=b_1,
\]
with $e_i\not\equiv 0 \pmod p$ since $\mathbf{e}\in \mathcal{F}$.

The second statement follows similarly, using  $\st_G(1)=\langle b,b^a,b^{a^2},\ldots,b^{a^{p^{m_1}-1}}\rangle$.
\end{proof}

With regards to the above result, it is also clear that if $\mathbf{e}\notin \mathcal{F}$, then the group~$G$ does not act spherically transitively on~$T$ and similarly $G$ is not $\sigma$-fractal.

\begin{lemma}\label{lem:super-strongly-fractal}
Let $G$ be a growing GGS-group such that for all $n$,
\[
\psi_1(\gamma_3(G_{n-1}))\ge \gamma_3(G_n)\times \overset{p^{m_n}}\dots\times \gamma_3(G_n).
\]
Then $\mathbf{e}\in \mathcal{F}$ and $\mathbf{e}$ is non-constant. Furthermore $G$ is super strongly $\sigma$-fractal.
\end{lemma}

\begin{proof}
The statement $\mathbf{e}\in\mathcal{F}$ is clear since $\gamma_3(G_1)$ equals $\langle [a_1,b_1,a_1],[a_1,b_1,b_1]\rangle^{G_1}$, and not $\langle [a_1^{\,p},b_1,a_1^{\,p}],[a_1^{\,p},b_1,b_1]\rangle^{G_1}$. Suppose that $\mathbf{e}$ is a constant vector $(c,\overset{p^{m_1}-1}\ldots,c)$ for some $c\not\equiv 0 \pmod p$. We proceed to show that $([a_1^{\,c},b_1, a_1^{\,c}],1,\ldots,1)\not\in \psi_1\big(\gamma_3(\st_G(1))\big)$. Since   by assumption
$\psi_1\big(\gamma_3(\st_G(1))\big)=\gamma_3(G_1)\times \overset{p^{m_1}}\dots\times \gamma_3(G_1)$, this yields the desired contradiction. 

To this end, observe that we may without loss of generality work modulo $\gamma_4(G_1)\times \overset{p^{m_1}}\dots\times \gamma_4(G_1)$. Note that $\psi_1\big(\gamma_3(\st_G(1))\big)$ is normally generated by the elements $[b,b^{a^i},b^{a^j}]$ for $i,j\in\{0,1,\ldots,p^{m_1}-1\}$ with $i\ne 0$, and modulo $\gamma_4(G_1)\times \overset{p^{m_1}}\dots\times \gamma_4(G_1)$ we have 
\begin{align*}
\psi_1( [b,b^{a^{i}},b])&\equiv(1,\overset{i-1}\ldots, 1, [a_1,b_1,a_1]^{2c},1,\ldots, 1,[a_1,b_1,b_1]^{-c})\\
\psi_1( [b,b^{a^{i}},b^{a^{i}}])&\equiv(1,\overset{i-1}\ldots, 1, [a_1,b_1,b_1]^c,1,\ldots, 1,[a_1,b_1,a_1]^{-2c})
\end{align*}
and 
\[
\psi_1( [b,b^{a^{i}},b^{a^{j}}])\equiv(1,\overset{i-1}\ldots, 1, [a_1,b_1,a_1]^{2c},1,\ldots, 1,[a_1,b_1,a_1]^{-2c})
\]
for $j\ne 0,i$. From the form of these generators, it follows that 
$([a_1^{\,c},b_1, a_1^{\,c}],1,\ldots,1)\not\in \psi_1\big(\gamma_3(\st_G(1))\big)$, as wanted.

For the final statement, it is clear that $\varphi_x(\st_G(1))=G_1$ for every $x\in X_1$. Then note that for $G_1$ the defining vector is $\mathbf{e}_1:=(e_1,\ldots,e_{p^{m_1}-1},1,\overset{p^{m_2}-p^{m_1}}\ldots,1)$. Using the fact that $\overline{m}$ is arbitrary and that $G_1$ is $\sigma$-fractal, it suffices to show that $\psi_1(\gamma_3(G_1))$ is a subdirect product of $G_2 \times \overset{p^{m_2}}\dots\times G_2$. Observe that 
\begin{align*}
\psi_1([b_1,a_1,a_1])&=(b_2^{-1}a_2^{e_1}b_2^{-1}, a_2^{e_2-2e_1}b_2, a_2^{e_1-2e_2+e_3},\ldots, a_2^{e_{p^{m_1}-3}-2e_{p^{m_1}-2}+e_{p^{m_1}-1}}, \\
&\qquad a_2^{e_{p^{m_1}-2}-2e_{p^{m_1}-1}},a_2^{e_{p^{m_1}-1}},1,\ldots, 1,b_2).
\end{align*}
From this it is clear that $\psi_1(\gamma_3(G_1))$ is a subdirect product of $G_2 \times \overset{p^{m_2}}\dots\times G_2$. 
\end{proof}

Now we determine  cases when the branching properties in the above lemma hold.

\begin{notation} 
Again following~\cite{DFG}, for $\mathbf{e}\in\mathcal{F}$, we define
\[
Y(\mathbf{e})=\{i\in\{1,\ldots, p^{m_1}-1\}\,\mid \,e_i\not\equiv 0\pmod p\}.
\]
For simplicity, we will just write  $Y$ in the sequel.
\end{notation}

The next result enables us to show that a large subfamily of growing GGS-groups are weakly branch.

\begin{lemma}\label{lem:G'}
Let $G$ be a growing GGS-group defined by $\mathbf{e}\in \mathcal{F}$, and suppose there exists $i\in Y$   such that $p^{m_1}-i\notin Y$. Then for all $n$,
    \[
\psi_1(\st_{G_{n-1}}(1)')= G_n'\times \overset{p^{m_n}}\dots\times G_n'. 
\]
\end{lemma}

\begin{proof}
Since $\overline{m}$ is arbitrary and $G$ is $\sigma$-fractal,  it suffices to consider the case $n=1$. Clearly $\psi_1(\st_{G}(1)')\le G_1'\times \overset{p^{m_1}}\dots\times G_1'$. As~$G$ is spherically transitive, it suffices to show that $1\times\overset{j}\dots \times 1\times G_1'\times 1\times\dots \times 1\le \psi_1(\st_{G}(1)')$ for some $j\in\{0,\ldots,p^{m_1}-1\}$.
To this end, we consider
\begin{align*}
    \psi_1([b, b^{a^i}])&=(1,\overset{i-1}\ldots, 1,[a_1^{\, e_{i}},b_1],1,\ldots,1,[b_1,a_1^{\, e_{p^{m_1}-i}}]).
\end{align*}
If $e_{p^{m_1}-i}=0$, we are done. So we suppose that $e_{p^{m_1}-i}\ne 0$. 
Write $k:=e_i$ and let $\mu\in\{1,\ldots, p^{m_2}-1\}$ be such that $k\mu\equiv 1 \pmod {p^{m_2}}$. Further, let $e_{p^{m_1}-i}=\lambda p^d$, for some $d\in\{1,\ldots,m_2-1\}$ with $\lambda\not\equiv 0 \pmod p$. 
Then, as
\begin{align*}
    \psi_1\big([b, b^{a^i}]\big)&=(1,\overset{i-1}\ldots, 1,[a_1^{\,k},b_1],1,\overset{p^{m_1}-i-1}\ldots,1,[b_1,a_1^{\,\lambda p^d}]),\\
     \psi_1\big([b^{\mu\lambda p^d}, b^{a^i}]^{a^{-i}}\big)&=(1,
     \ldots,1,[b_1^{\,\mu\lambda p^d},a_1^{\,\lambda p^d}],1,\overset{i-1}\ldots,1,[a_1^{\,\lambda p^d},b_1]),\\
     \psi_1\big([b^{\mu\lambda p^d}, (b^{a^i})^{\mu\lambda p^d}]^{a^{-2i}}\big)&= (1,
     \ldots,1,[b_1^{\,\mu\lambda p^d},a_1^{\,\mu\lambda^2 p^{2d}}],1,\overset{i-1}\ldots,1,[a_1^{\,\lambda p^d},b_1^{\,\mu\lambda p^d}],1,\overset{i-1}\ldots,1),\\
       \psi_1\big([b^{\mu^2\lambda^2 p^{2d}}, (b^{a^i})^{\mu\lambda p^d}]^{a^{-3i}}\big)&= (1, \ldots,1,[b_1^{\,\mu^2\lambda^2 p^{2d}},a_1^{\,\mu\lambda^2 p^{2d}}],1,\overset{i-1}\ldots,1,[a_1^{\,\mu\lambda^2 p^{2d}},b_1^{\,\mu\lambda p^{d}}],1,\overset{2i-1}\ldots,1),\\
     &\,\,\,\vdots\\
       \psi_1\big([b^{\mu^{\ell}\lambda^{\ell} p^{\ell d}}, (b^{a^i})^{\mu^{\ell}\lambda^{\ell} p^{\ell d}}]^{a^{-2\ell i}}\big)&=(1, \ldots,1,[b_1^{\,\mu^\ell\lambda^{\ell} p^{\ell d}},a_1^{\,\mu^{\ell}\lambda^{\ell+1} p^{(\ell+1) d}}],1,\overset{i-1}\ldots,1,\\
       &\qquad\qquad\qquad\qquad\qquad\qquad  [a_1^{\,\mu^{\ell -1}\lambda^\ell p^{\ell d}},b_1^{\,\mu^\ell\lambda^\ell p^{\ell d}}],1,\overset{(2\ell-1) i-1}\ldots,1),
       \end{align*}
 where $\ell$ is the minimal integer such that $m_2$ divides $(\ell+1) d$, we see that
\begin{align*}
   &\psi_1\big([b, b^{a^i}]\cdot [b^{\mu\lambda p^d}, b^{a^i}]^{a^{-i}} \cdot [b^{\mu\lambda p^d}, (b^{a^i})^{\mu\lambda p^d}]^{a^{-2i}}  \cdots  [b^{\mu^{\ell}\lambda^{\ell} p^{\ell d}}, (b^{a^i})^{\mu^{\ell}\lambda^{\ell} p^{\ell d}}]^{a^{-2\ell i}}\big)=\\ &\qquad\qquad\qquad\qquad\qquad\qquad\qquad\qquad\qquad\qquad\qquad\qquad\qquad\qquad (1,\overset{i-1}\ldots, 1,[a_1^{\, k},b_1],1,\ldots,1).
\end{align*}
The result then follows. 
\end{proof}

\begin{theorem}
Let $G$ be a growing GGS-group, defined by $\mathbf{e}\in\mathcal{F}$. Then 
for all $n$,
    \[
\psi_1(G_{n-1}'')\ge G_n''\times \overset{p^{m_n}}\dots\times G_n''. \]
In particular, the group $G$ is weakly branch. 
\end{theorem}

\begin{proof} As usual, since $\overline{m}$ is arbitrary, it suffices to prove the result for $n=1$.

We follow the proof of \cite[Thm.~1(i)]{DFG}.
First 
we suppose that there exists an $i\in Y$ such that $p^{m_1}-i\notin Y$. The result then follows from Lemma~\ref{lem:G'}.
So suppose that for all $i\in Y$ we have $p^{m_1}-i\in Y$. We fix some $i\in Y$. 

Next, observe that $\text{St}_G(1)=\text{st}_G(x)$ for any $x\in X_1$. Then from Lemma~\ref{lem:fractal}, for any $g_1,g_2\in G_1$, there exist $h_1,h_2\in \text{St}_G(1)$ such that
\begin{align*}
    \psi_1(h_1)&=(*,\overset{i-1}\ldots,*,g_1,*,\ldots,*)\\
    \psi_1(h_2)&=(*,\overset{i-1}\ldots,*,g_2,*,\ldots,*),
\end{align*}
where $*$ are unspecified elements. One computes that
\[
\psi_1\Big(\big[[b,b^{a^i}]^{h_1},[b^{a^i},b^{a^{2i}}]^{h_2}\big]\Big)=(1,\overset{i-1}\ldots,1,\big[[a_1^{e_i},b_1]^{g_1},[b_1,a_1^{e_{p^{m_1}-i}}]^{g_2}\big],1,\ldots,1).
\]
As 
\[
G_1'=\langle [a_1^{\,e_i},b_1]^{g_1}\mid g_1\in G_1\rangle=\langle [b_1,a_1^{\,e_{p^{m_1}-i}} ]^{g_2}\mid g_2\in G_1\rangle,
\]
the result follows.
\end{proof}

\begin{lemma}\label{lem:branch_gamma_3}
Let $G$ be a growing GGS-group defined by a non-constant $\mathbf{e}\in \mathcal{F}$ and let $i\in Y$. 
If there exists $k$ 
such that $e_k^2-e_{k-i}e_{k+i}\not\equiv 0 \pmod p$, then
for all~$n$,
    \[
\psi_1\big(\gamma_3(\st_{G_{n-1}}(1))\big)= \gamma_3(G_n)\times \overset{p^{m_n}}\dots\times \gamma_3(G_n). \]
\end{lemma}

\begin{proof}
As usual, since $\overline{m}$ is arbitrary and $G$ is $\sigma$-fractal,  it suffices to consider the case $n=1$. Clearly $\psi_1\big(\gamma_3(\st_G(1))\big)\le \gamma_3(G_1)\times \overset{p^{m_1}}\dots\times \gamma_3(G_1)$, 
so as before it suffices to show that $1\times\overset{j}\dots \times 1\times \gamma_3(G_1)\times 1\times\dots \times 1\le \psi_1\big(\gamma_3(\st_G(1))\big)$ for some $j\in\{0,\ldots,p^{m_1}-1\}$. We proceed as in \cite[Lem.~3.2]{FAZR2}.

Let $i\in Y$. 
Setting
\[
g_k=(b^{a^{p^{m_1}-k+1}})^{e_k}(b^{a^{p^{m_1}-k-i+1}})^{-e_{k-i}},
\]
we have
\begin{align*}
    \psi_1(g_k)&=(a_1^{\,e_k^{\,2}-e_{k-i}e_{k+i}},*,\ldots,*, a_1^{\,e_{k-i}e_k-e_ke_{k-i}},*,\overset{i-1}\ldots,*)\\
    &=(a_1^{\,e_k^{\,2}-e_{k-i}e_{k+i}},*,\ldots,*, 1,*,\overset{i-1}\ldots,*),
\end{align*}
where $*$ represents unspecified elements. Now we can find a suitable power~$g$ of~$g_k$ such that
\[
\psi_1(g)=(a_1,*,\ldots,*, 1,*,\overset{i-1}\ldots,*).
\]
Let $\mu\in\{1,\ldots,p^{m_2}-1\}
$ be such that $\mu e_i\equiv 1 \pmod {p^{m_2}}$. Then
\[
\psi_1\big(b^a(b^{a^{p^{m_1}-2i+1}})^{-\mu e_{p^{m_1}-i}} \big)= (b_1a_1^{\,-\mu e_{2i}e_{p^{m_1}-i}}
,*,\ldots,*,1,*,\overset{i-1}\ldots,*),
\]
and it is straightforward now to show that $\gamma_3(G_1)\times 1\times \cdots \times 1\le \psi_1\big(\gamma_3(\st_G(1))\big)$,
as \begin{align*}
\psi_1\big([(b^{a^{-(i-1)}})^\mu, b^a,g]\big)&= ([a_1,b_1,a_1]
,1,\ldots,1),\\
\psi_1\big([(b^{a^{-(i-1)}})^\mu, b^a,b^a(b^{a^{p^{m_1}-2i+1}})^{-\mu e_{p^{m_1}-i}} ]\big)&= ([a_1,b_1,b_1a_1^{\,-\mu e_{2i}e_{p^{m_1}-i}}
]
,1,\ldots,1).\qedhere
\end{align*}
\end{proof}

\begin{lemma}\label{lem:branching-subgroup}
Let $G$ be a growing GGS-group satisfying the conditions given in
Lemma~\ref{lem:branch_gamma_3}. If further $e_i\not\equiv e_{p^{m_1}-i} \pmod p$, then for all $n$,
    \[
\psi_1(\st_{G_{n-1}}(1)')= G_n'\times \overset{p^{m_n}}\dots\times G_n'. 
\]
\end{lemma}

\begin{proof}
As in the previous proof, it suffices to show that $(1,\overset{j}\ldots,1,[a_1,b_1],1,\ldots,1)\in \psi_1(\st_G(1)')$ for some $j\in\{0,\ldots,p^{m_1}-1\}$. To this end, we consider
\begin{align*}
    \psi_1([b, b^{a^i}])&=(1,\overset{i-1}\ldots, 1,[a_1^{\, e_{i}},b_1],1,\ldots,1,[b_1,a_1^{\, e_{p^{m_1}-i}}]).
\end{align*}

In view of Lemma~\ref{lem:G'}, we may 
suppose that  $p^{m_1}-i\in Y$. 
Further by Lemmata~\ref{lem:first-over-derived}  and~\ref{lem:branch_gamma_3}, we may assume without loss of generality that $e_i=1$.
Now write $\mu:=e_{p^{m_1}-i}\not\equiv 1 \pmod p$, and 
suppose $i=\lambda p^{\tau}$ for some $\lambda\not\equiv 0 \pmod p$ and $\tau\in\{0,1,\ldots, m_1-1\}$.
Then, working modulo $\gamma_3(G_1)\times\overset{p^{m_1}}\ldots\times \gamma_3(G_1)$, we obtain
\begin{align*}
    \psi_1\big([b, b^{a^i}]\big)&\equiv (1,\overset{i-1}\ldots, 1,[a_1,b_1],1,\ldots,1,[b_1,a_1]^{\mu})\\
     \psi_1\big(([b, b^{a^i}]^{a^{p^{m_1}-i}})^{\mu}\big)&\equiv (1,\overset{p^{m_1}-i-1}\ldots,1,[b_1,a_1]^{\mu^2},1,\overset{i-1}\ldots, 1,[a_1,b_1]^{\mu})\\
      \psi_1\big(([b, b^{a^i}]^{a^{p^{m_1}-2i}})^{\mu^2}\big)&\equiv (1,\overset{p^{m_1}-2i-1}\ldots,1,[b_1,a_1]^{\mu^3},1,\overset{i-1}\ldots, 1,[a_1,b_1]^{\mu^2},1,\overset{i}\ldots, 1)\\
      &\,\,\,\vdots\\
       \psi_1\big(([b, b^{a^i}]^{a^{i}})^{\mu^{\kappa p^{m_1-\tau}-1}}\big)&\equiv (1,\overset{i-1}\ldots,1,[b_1,a_1]^{\mu^{\kappa p^{m_1-\tau}}},1,\overset{i-1}\ldots, 1,[a_1,b_1]^{\mu^{\kappa p^{m_1-\tau}-1}},1,\overset{p^{m_1}-2i}\ldots, 1)
\end{align*}
where $\kappa\in\mathbb{Z}/p^{m_1}\mathbb{Z}$ is such that $\kappa\lambda\equiv 1 \pmod {p^{m_1}}$, and thus
\begin{align*}
&\psi_1\big( [b, b^{a^i}]([b, b^{a^i}]^{a^{p^{m_1}-i}})^{\mu} ([b, b^{a^i}]^{a^{p^{m_1}-2i}})^{\mu^2}\cdots ([b, b^{a^i}]^{a^{i}})^{\mu^{\kappa p^{m_1-\tau}-1}}\big)\equiv\\
&\qquad\qquad\qquad(1,\overset{i-1}\ldots,1, [b_1,a_1]^{\mu^{\kappa p^{m_1-\tau}-1}}, 1,\ldots, 1) \pmod {\gamma_3(G_1)\times\overset{p^{m_1}}\ldots\times \gamma_3(G_1)}.
\end{align*}
As $\mu\not\equiv 0 \pmod p$, it follows that $\mu^{\kappa p^{m_1-\tau}-1}\not\equiv 0 \pmod p$. The result now follows making use of Lemma~\ref{lem:branch_gamma_3}.
\end{proof}

As seen in the proof of Lemma~\ref{lem:first-over-derived}, we have $\st_G(1)'\le \gamma_3(G)$ for $G$ a growing GGS-group. The following is immediate. 

\begin{lemma}
\label{lem:derived-stab-in-gamma-3}
Let $G$ be a growing GGS-group such that for all $n$,
\[
\psi_1(G_{n-1}')\ge \psi_1(\st_{G_{n-1}}(1)')= G_n'\times \overset{p^{m_n}}\dots\times G_n'.
\]
Then $\psi_1(\gamma_3(G_{n-1}))\ge G_n'\times \overset{p^{m_n}}\dots\times G_n'$  for all $n$.
\end{lemma}

For $\mathbf{e}\in\mathcal{F}$, recall the set
\[
Y=Y(\mathbf{e})=\{i\in\{1,\ldots, p^{m_1}-1\}\,\mid \,e_i\not\equiv 0\pmod p\}
\]
and, as in~\cite{DFG}, we define
\[
t(\mathbf{e})=\max\{s\ge 0\,\mid  \,i \equiv 0 \pmod {p^s}\quad\text{for all }i\in Y\}.
\]
Similarly,  we  write  $t=t(\mathbf{e})$ for brevity.
Next, we set
\[
\mathcal{E}=\{\mathbf{e}\in\mathcal{F}\,\mid \, e_{ip^t}\equiv e_{jp^t} \pmod p \quad\text{for all } i,j\in\{1,\ldots, p^{m_1-t}-1\}\}.
\]

\begin{theorem}\label{thm:branching-gamma-3}
Let $G$ be a growing GGS-group defined by $\mathbf{e}\in\mathcal{F}\backslash \mathcal{E}$, and suppose that  $i\in Y$ if and only if $p^{m_1}-i\in Y$. Then 
for all $n$,
    \[
\psi_1\big(\gamma_3(\st_{G_{n-1}}(1))\big)= \gamma_3(G_n)\times \overset{p^{m_n}}\dots\times \gamma_3(G_n). \]
\end{theorem}

\begin{proof}
Let $i\in Y$. 
By Lemma~\ref{lem:branch_gamma_3}, we may assume that  $e_k^2-e_{k-i}e_{k+i}\equiv 0 \pmod p$ for all~$k$ such that the terms $e_{k-i},e_k,e_{k+i}$ are defined.

\underline{Case 1:} Suppose $Y=\{p^t,2p^t,\ldots,p^{m_1}-p^t\}$. We proceed as in \cite[Thm.~2.11]{DFG}, by first observing that the  statement is true if $t=m_1-1$ and $p=3$. In this case, since $\mathbf{e}\notin\mathcal{E}$, we have either $e_{3^{m_1-1}}\equiv 1 \pmod 3$ and $e_{2\cdot 3^{m_1-1}}\equiv 2 \pmod 3$, or vice versa. We will assume that  $e_{3^{m_1-1}}\equiv 1 \pmod 3$ and $e_{2\cdot 3^{m_1-1}}\equiv 2 \pmod 3$; the other case is similar. For convenience, write $\alpha:=e_{ 3^{m_1-1}}$ and $\beta:=e_{2\cdot 3^{m_1-1}}$. By abuse of notation, let $\alpha^{-1},\beta^{-1}\in\{1,\ldots,3^{m_2}-1\}$ be the respective inverses of $\alpha$ and $\beta$ modulo~$3^{m_2}$. Consider
\begin{align*}
    \psi_1\big([b^{a^{2\cdot 3^{m_1-1}}},b, b(b^{a^{ 3^{m_1-1}}})^{-\alpha^{-1}\beta}]\big)&=(1,\ldots,1,[a_1^{\,\alpha},b_1,b_1a_1^{\, -\alpha^{-1}\beta^2}]),\\
     \psi_1\big([(b^{a^{ 3^{m_1-1}}})^{\beta^{-1}\alpha},b,(b^{a^{2\cdot 3^{m_1-1}}})^{\beta^{-1}\alpha} b^{-1}]\big)&=(1,\ldots,1,[a_1^{\,\alpha},b_1,a_1^{\, \beta^{-1}\alpha^2} b_1^{-1}]).
\end{align*}
Note that $\beta^{-1}\alpha^2-\alpha^{-1}\beta^2\not\equiv 0\pmod 3$. Indeed, if $\beta^{-1}\alpha^2\equiv\alpha^{-1}\beta^2 \pmod 3$, then we obtain $\alpha\equiv \beta \pmod 3$, a contradiction.
Therefore  $$
\langle b_1a_1^{\, -\alpha^{-1}\beta^2}, a_1^{\, \beta^{-1}\alpha^2} b_1^{-1}\rangle=\langle b_1a_1^{\, -\alpha^{-1}\beta^2},a_1^{\, \beta^{-1}\alpha^2-\alpha^{-1}\beta^2}\rangle=G_1,
$$
and it follows that $1\times\cdots\times 1\times \gamma_3(G_1)\le \psi_1\big(\gamma_3(\st_{G}(1))\big)$. Thus, the result follows in this case.

So we suppose that $(p,t)\ne (3,m_1-1)$.
Recall our assumption that $e_k^2-e_{k-i}e_{k+i}\equiv 0 \pmod p$ for all~$k$ such that the terms $e_{k-i},e_k,e_{k+i}$ are defined. Then, writing $\lambda:=e_{2p^t}e_{p^t}^{\,-1} \in(\mathbb{Z}/p^{m_2}\mathbb{Z})^*$, this inductively implies
\begin{equation}\label{eq:lambda}
e_{np^t}\equiv \lambda ^{n-1}e_{p^t} \pmod p
\end{equation}
for $n\in\{1,\ldots,p^{m_1-t}-1\}$. Note that $\lambda\not\equiv 1 \pmod p$ as $\mathbf{e}\notin\mathcal{E}$. 

Let $\mu\in\{1,\ldots,p^{m_2}-1\}$ be such that $\mu e_{p^t}\equiv 1\pmod{p^{m_2}}$. Observe that
\[
\psi_1\big([b,b^{a^{p^t}},b^{a^{p^t}} (b^{a^{-p^t}})^{-\mu e_{p^{m_1}-p^t}} ]\big)=(1,\overset{p^t-1}\ldots,1,[a_1^{\,e_{p^t}},b_1,b_1a_1^{\,-\mu e_{2p^t}e_{p^{m_1}-p^t}}],1,\ldots,1)
\]
and also
\[
\psi_1\big(b^{e_{p^{m_1}-3p^t}}(b^{a^{2p^t}})^{-e_{p^{m_1}-p^t}}\big)=(*,\overset{p^t-1}\ldots,*, a_1^{\,e_{p^t}e_{p^{m_1}-3p^t}\,-\,e_{p^{m_1}-p^t}^{\,2}},*,\ldots,*,1,*,\overset{p^t-1}\ldots,*).
\]
For convenience, we write
$g:=b^{e_{p^{m_1}-3p^t}}(b^{a^{2p^t}})^{-e_{p^{m_1}-p^t}}$. Then, viewing $\lambda^{-1}$ as an integer between 1 and $p^{m_2}-1$, we obtain
\[
\psi_1\big([(b^{a^{-p^t}})^{\lambda^{-1}},b^{a^{p^t}},g]\big)=(1,\overset{p^t-1}\ldots,1,[a_1^{\,e_{p^t}},b_1,a_1^{\,e_{p^t}e_{p^{m_1}-3p^t}\,-\,e_{p^{m_1}-p^t}^{\,2}}],1,\ldots,1).
\]
From Equation~\eqref{eq:lambda}, we have
\begin{align*}
e_{p^t}e_{p^{m_1}-3p^t}\,-\,e_{p^{m_1}-p^t}^{\,2} &\equiv \lambda^{p^{m_1-t}-4}e_{p^t}^{\,2}\,-\,\lambda^{2(p^{m_1-t}-2)} e_{p^t}^{\,2} \pmod p,\\
&\equiv \lambda^{p^{m_1-t}-4}e_{p^t}^{\,2}(1\,-\,\lambda^{p^{m_1-t}}) \quad\pmod p.
\end{align*}
As none of these three factors is congruent to 0 modulo $p$, it follows that
\[
e_{p^t}e_{p^{m_1}-3p^t}\,-\,e_{p^{m_1}-p^t}^{\,2}\not\equiv 0 \pmod p.
\]
Hence 
\[
G_1=\langle a_1^{\,e_{p^t}e_{p^{m_1}-3p^t}\,-\,e_{p^{m_1}-p^t}^{\,2}}, b_1a_1^{\,-\mu e_{2p^t}e_{p^{m_1}-p^t}}\rangle,
\]
and since
\[
G_1'=\langle [a_1^{\,e_{p^t}},b_1]\rangle^{G_1},
\]
we obtain
\[
\gamma_3(G_1)=\big\langle  [a_1^{\,e_{p^t}},b_1,a_1^{\,e_{p^t}e_{p^{m_1}-3p^t}\,-\,e_{p^{m_1}-p^t}^{\,2}}], [a_1^{\,e_{p^t}},b_1,b_1a_1^{\,-\mu e_{2p^t}e_{p^{m_1}-p^t}}]\big\rangle^{G_1}.
\]
Thus we have
\[
1\times\overset{p^t-1}\dots \times 1\times \gamma_3(G_1)\times 1\times\dots \times 1\le \psi_1\big(\gamma_3(\st_G(1))\big).
\]

\underline{Case 2:} Suppose $Y\subset\{p^t,2p^t,\ldots,p^{m_1}-p^t\}$ is a proper subset.  We proceed analogously to the proof of \cite[Thm.~2.9]{DFG}. Let $h$ and $\ell$ be such that $hp^t\notin Y$ but  $\ell p^t, (h-\ell)p^t\in Y$; here $-kp^t$ stands for $p^{m_1}-kp^t$, for a positive integer~$k$. For convenience, we write 
$$
\lambda:=e_{\ell p^t},\quad q:=e_{p^{m_1}-hp^t},\quad z:=e_{p^{m_1}-\ell p^t},\quad  y:=e_{p^{m_1}-(h-\ell)p^t}.
$$
Note that $q\equiv 0\pmod p$ and by our hypothesis $y,z\not\equiv 0\pmod p$. Also it is possible that $y=z$.
As $q\equiv 0\pmod p$ and $a_1$ has order~$p^{m_2}$, there is a minimal positive integer $r$ such that $a_1^{\,q^r}=1$.

For an integer $n\ge 0$, set
\begin{align*}
g_{2n}&:=\big[b^{\lambda^{-n}z^n},(b^{a^{\ell p^t}})^{\lambda^{-n}z^n}, (b^{a^{hp^t}})^{q^{2n}y^{-2n-1}}\big]^{a^{-2n\ell p^t}}
\end{align*}
and
\begin{align*}
k_{2n+1}&:=\big[(b^{a^{\ell p^t}})^{\lambda^{-n}z^n},b^{\lambda^{-n-1}z^{n+1}}, (b^{a^{hp^t}})^{q^{2n+1}y^{-2n-2}}\big]^{a^{-3n\ell p^t}}.
\end{align*}
Then
\begin{align*}
    \psi_1(g_{2n})&=(1,\ldots,1,
    [b_1^{\,\lambda^{-n}z^n},a_1^{\,\lambda^{-n} z^{n+1}},a_1^{\,q^{2n+1}y^{-2n-1}}]
,1,\overset{\ell p^t-1}\ldots,1,\\
&\qquad\qquad\qquad\qquad\qquad\qquad\qquad\qquad [a_1^{\, \lambda^{-n+1}z^n},b_1^{\,\lambda^{-n}z^n},a_1^{\,q^{2n}y^{-2n}}],1,\overset{(2n-1)\ell p^t-1}\ldots,1),\\
    \psi_1(k_{2n+1})&=(1,\ldots,1,[a_1^{\,\lambda^{-n}z^{n+1}},b_1^{\,\lambda^{-n-1}z^{n+1}},a_1^{\,q^{2n+2}y^{-2n-2}}],1,\overset{\ell p^t-1}\ldots,1,\\
   &\qquad\qquad\qquad\qquad\qquad\qquad\qquad\qquad 
    [b_1^{\,\lambda^{-n}z^n},a_1^{\, \lambda^{-n}z^{n+1}},a_1^{\,q^{2n+1}y^{-2n-1}}],1,\overset{2n\ell p^t-1}\ldots,1).
\end{align*}
Note that the $\lambda^{-1}$ (respectively $y^{-1}$) in the exponent of $b$ and $b_1$ is viewed as an integer, which upon reduction modulo~$p^{m_2}$ equals the inverse of $\lambda$ (respectively $y$) in $\mathbb{Z}/p^{m_2}\mathbb{Z}$.

Define
\[
g:=\begin{cases}
g_0k_1^{-1}g_2k_3^{-1}\cdots k_{r-1}^{-1} &\text{if $r$ is even},\\
g_0k_1^{-1}g_2k_3^{-1}\cdots g_{r-1} &\text{if $r$ is odd}.
\end{cases}
\]
A direct computation yields
\[
g:=\psi_1^{-1}\big((1,\overset{\ell p^t-1}\ldots,1,[a_1^{\,\lambda},b_1,a_1],1,\ldots,1)\big)\in\gamma_3(G).
\]
Hence, writing $N=\langle [a_1^{\,\lambda},b_1,a_1]\rangle^{G_1}$,  we have by conjugation that
\[
N\times\overset{p^{m_1}}{\ldots}\times N\le\psi_1(\gamma_3(G)).
\]

Finally, we consider
\begin{align*}
    \psi_1\Big(\big[b,b^{a^{\ell p^t}}, b^{a^{\ell p^t}}(b^{a^{-\ell p^t}})^{\lambda^{-1}z}\big]\Big)
    &=(1,\overset{\ell p^t-1}\ldots,1,[a_1^{\,\lambda},b_1,b_1a_1^{\,\lambda^{-1}ze_{2\ell p^t}}],1,\ldots,1
    ).
\end{align*}
Since $G_1=\langle a_1,b_1a_1^{\,\lambda^{-1}ze_{2\ell p^t}}\rangle$, we obtain $\gamma_3(G_1)=\langle [a_1^{\,\lambda},b_1,a_1],[a_1^{\,\lambda},b_1,b_1a_1^{\,\lambda^{-1}ze_{2\ell p^t}}]\rangle^{G_1}$, and the result follows for $n=1$. 

\smallskip

For $n>1$, the result follows in a straightforward manner from the definition of $b_1$, in particular as $\psi_1(b_1)=(a^*,\ldots,a^*,a^{e_{p^t}},a^*,\ldots,a^{e_{p^{m_1}-1}},1,\ldots,1,b_2)$, so the component in position $p^{m_2}-p^t$ is trivial.
\end{proof}

\begin{corollary}\label{cor:periodic}
Let $G$ be a growing GGS-group with defining vector $\mathbf{e}$ satisfying~\eqref{eq:periodicity}.
Then 
for all $n$,
    \[
\psi_1\big(\gamma_3(\st_{G_{n-1}}(1))\big)= \gamma_3(G_n)\times \overset{p^{m_n}}\dots\times \gamma_3(G_n). \]
\end{corollary}

\begin{proof}
This is akin to  \cite[Cor.~2.12]{DFG}. 
Indeed, if any of the conditions in Lemmas~\ref{lem:G'} and \ref{lem:branch_gamma_3} are satisfied, we are done. So we suppose otherwise. Then by Theorem~\ref{thm:branching-gamma-3}, it suffices to show that $\mathbf{e}\notin\mathcal{E}$. Suppose on the contrary that $\mathbf{e}\in\mathcal{E}$. Then $e_{ip^t}\equiv e_{jp^t} \pmod p$ for all $1\le i,j\le p^{m_1-t}-1$. However, since $\mathbf{e}$ satisfies \eqref{eq:periodicity}, it follows from
\[
\sum_{k=1}^{p^{m_1-t}-1} e_{kp^t} \equiv (p^{m_1-t}-1)e_{p^t}\equiv 0 \pmod p
\]
that $e_{p^t}\equiv 0 \pmod p$, which contradicts the definition of $Y$ and $t$. Hence the result.
\end{proof}

\begin{theorem}\label{lem:branch-gamma-3}
Let $G$ be a growing GGS-group such that for all $n$,
\[
\psi_1(\gamma_3(G_{n-1}))\ge \gamma_3(G_n)\times \overset{p^{m_n}}\dots\times \gamma_3(G_n).
\]
Then for all $n$,
\[
\psi_n(\rist_G(n))\ge\langle b_n^{\,p^{m_{n+1}}}\rangle^{G_n} \gamma_3(G_n)\times \overset{p^{m_1+\cdots+m_n}}\dots \times \langle b_n^{\,p^{m_{n+1}}}\rangle^{G_n} \gamma_3(G_n).
\]
In particular, the group~$G$ is branch. 

Further if
$
\psi_1(G_{n-1}')\ge G_n'\times \overset{p^{m_n}}\dots\times G_n'$ for all  $n$,
then 
\[
\psi_n(\rist_G(n))\ge\langle b_n^{\,p^{m_{n+1}}}\rangle G_n'\times \overset{p^{m_1+\cdots+m_n}}\dots \times \langle b_n^{\,p^{m_{n+1}}}\rangle G_n'
\]
for all $n$.
\end{theorem}

\begin{proof}
For $n=1$, this follows from the definition of~$b_1$, in particular using the fact that the order of~$a_1$ is~$p^{m_2}$. 
For $n\ge 2$, as $\varphi_{u}(\rist_G(n))\le \rist_{G_{n-1}}(1)$ for any $(n-1)$st-level vertex~$u$, the result follows from the $n=1$ case, since $\overline{m}$ is arbitrary.
The fact that $\rist_G(n)$ has finite index in~$G$ follows from  Lemma~\ref{lem:first-over-derived}. Indeed, we have that $|G:\langle b^{p^{m_{n+1}}}\rangle G'|<\infty$ and $|\langle b^{p^{m_{n+1}}}\rangle G':\langle b^{p^{m_{n+1}}}\rangle^G \gamma_3(G)|<\infty$, and hence $|G:\langle b^{p^{m_{n+1}}}\rangle^G \gamma_3(G)|<\infty$. The result then follows from the inequality 
$|\psi_n(\st_G(n)):\psi_n(\rist_G(n))|\le |G_n:\langle b_n^{\,p^{m_{n+1}}}\rangle^{G_n} \gamma_3(G_n)|^{p^{m_1+\cdots+m_n}}$.

The final statement  is proved similarly.
\end{proof}

\begin{corollary}
\label{cor:saturated}
Let $G$ be a growing GGS-group, defined by $\mathbf{e}\in\mathcal{F}$ which is non-constant modulo~$p$, and such that for all $n$,
\[
\psi_1(G_{n-1}')\ge G_n'\times \overset{p^{m_n}}\dots\times G_n'.
\]
Then $\textup{Aut }G=N_{\textup{Aut }T}(G)$.
\end{corollary}

\begin{proof}
By~\cite[Thm.~7.5]{LN}, it suffices to show that, for any $n$ there exists a subgroup $L_n\le \text{St}_G(n)$ that is characteristic in $G$ and $\varphi_v(L_n)$ acts spherically transitively on $T_v$ for all  $n$th level vertices $v$. Set $L_0=G$ and $L_{n+1}=L_n'$. Note that the restriction of $G'$ on the first-level vertices of the tree is the whole group $G_1$. Hence it follows by induction that each restriction $\varphi_v(L_n)$  is the whole group $G_n$ and thus acts spherically transitively on $T_v$.
\end{proof}

\subsection{Congruence subgroup properties}\label{sec:CSP}

Here we investigate various congruence subgroup properties for growing GGS-groups.

\begin{proposition}\label{pro:no-CSP}
Let $G$ be a growing GGS-group.
Then~$G$ does not have the congruence subgroup property.
\end{proposition}

\begin{proof}
It follows from Proposition~\ref{abelianization} that we have quotients of non-$p$-power order. Hence the result follows.
\end{proof}

The following lemma highlights another difference compared to the GGS-groups acting on the $p$-regular rooted tree.
\begin{lemma}
\label{lem:second-stab}
Let~$G$ be a growing GGS-group. Then $\st_G(n)\not\le G'$ for all $n$.
\end{lemma}

\begin{proof}
From Proposition~\ref{abelianization}, we know that $G'$ has infinite index in~$G$. As $\st_G(n)$ has finite index in~$G$ for all $n$, the result follows.
\end{proof}

Next, among the growing GGS-groups, we find  the first examples of branch groups satisfying the $p$-congruence subgroup property but not the congruence subgroup property. See \cite{pcongruence}, \cite{PR} and~\cite{DNT} for weakly branch, but not branch, examples.

\begin{lemma}\label{lem:stabdescription}
Let $G$ be a growing GGS-group with defining vector~$\mathbf{e}$ such that
\[
\psi_1(G_{n-1}')\ge G_n'\times \overset{p^{m_n}}\dots\times G_n'
\]
and $\st_{G_{n-1}}(2)\le \langle b_{n-1}^{\,p^{m_n}}\rangle G_{n-1}'$ for all $n$.
Then for all $n\ge 2$,
\[
\st_G(n)\le\big\langle \psi_{n-2}^{\,-1}(G_{n-2}'\times \overset{p^{m_1+\cdots+m_{n-2}}}\ldots\times G_{n-2}'), b^{\,p^{m_{n-1}}}  \big\rangle^G.
\]
\end{lemma}

\begin{proof}
Since $\psi_{n-2}(\st_G(n))\le \st_{G_{n-2}}(2)\times\overset{p^{m_1+\cdots+m_{n-2}}}\ldots\times \st_{G_{n-2}}(2)$, the result is immediate from our hypotheses. 
\end{proof}

\begin{example}
An example of a growing GGS-group satisfying the conditions of the above lemma is the group defined by the vector $(1,-1,0,\overset{p^{m_1}-3}\ldots,0 )$, which one could refer to as a growing Gupta-Sidki group.  We see that for this group $\psi_1(G_{n-1}')\ge G_n'\times \overset{p^{m_n}}\dots\times G_n'$ by Lemma~\ref{lem:G'}. The additional hypothesis then follows 
as in \cite[Thm.~2.13]{FAZR2}, where we see that $g\in \text{St}_{G_{n-1}}(2)$ if and only if,
writing $$g\equiv b_{n-1}^{r_0}\Big(b_{n-1}^{\,a_{n-1}}\Big)^{r_1}\cdots \Big(b_{n-1}^{\,a_{n-1}^{p^{m_n}-1}}\Big)^{r_{p^{m_n}-1}}\pmod {\st_{G_{n-1}}(1)'}$$
we have that $(r_0,r_1,\ldots,r_{p^{m_n}-1} )$ is in the kernel of the circulant $p^{m_n}\times p^{m_n}$ matrix
\[
\left(
\begin{array}{cccccc}
    1 & -1 & 0& 0& \cdots&  0 \\
   0& 1 & -1 & 0& \cdots&  0 \\
    \vdots & & & & &  \vdots \\
     0& \cdots&  0  & 1 & -1 &   0 \\
  0& \cdots&  0  &   0& 1 & -1 
\end{array}\right).
\]
As this kernel is $\langle (1,\overset{p^{m_n}}\ldots,1)\rangle$, it follows that $g\in \text{St}_{G_{n-1}}(2)$ if and only if 
\begin{align*}
    g&\equiv \bigg(b_{n-1}b_{n-1}^{\,a_{n-1}}\cdots b_{n-1}^{\,a_{n-1}^{p^{m_n}-1}}\bigg)^r\pmod {\st_{G_{n-1}}(1)'}\\
    &\equiv b_{n-1}^{\,rp^{m_n}}\pmod {G_{n-1}'}
\end{align*}
for $r\in\mathbb{Z}$.
\end{example}

\begin{theorem}
Let $G$ be a growing GGS-group satisfying the conditions in Lemma~\ref{lem:stabdescription}, and with defining vector $\mathbf{e}$ being non-constant modulo~$p$.
Then $G$ has the $p$-congruence subgroup property.
\end{theorem}

\begin{proof}
We need to show that every normal subgroup of $p$-power index  contains $\st_G(n)$ for some~$n$. Fix $N\trianglelefteq G$ of index $p^k$ for some~$k\in\mathbb{N}$. Since $G/N$ is a $p$-group of order~$p^k$, it has exponent at most $p^k$ and nilpotency class at most~$k$ and so in particular, the subgroup~$N$ contains $\gamma_k(G)$, $a^{p^k}$, and $b^{p^k}$. 
By Lemma~\ref{lem:derived-stab-in-gamma-3}, we have $\gamma_3(G)\ge \psi_1^{-1}(G_1'\times \overset{p^{m_1}}\dots\times G_1')$, and as
\[
\gamma_4(G)\ge \psi_1^{-1}(\gamma_3(G_1)\times \overset{p^{m_1}}\dots\times \gamma_3(G_1))\ge \psi_2^{-1}(G_2'\times \overset{p^{m_1+m_2}}\dots\times G_2'),
\]
we deduce that $\gamma_k(G)$ contains $\psi_{k-2}^{-1}(G_{k-2}' \times \overset{p^{m_1+\cdots+m_{k-2}}}\ldots \times G_{k-2}')$;  recall that $\psi_1(G')$ is a subdirect product of $G_1\times\overset{p^{m_1}}\ldots\times G_1$. Therefore by Lemma~\ref{lem:stabdescription} and noting that $m_k\ge k$,  we obtain that $N$ contains $\st_G(k+1)$.
\end{proof}

\begin{lemma}\label{lem:weak-CSP}
Let $G$ be a growing GGS-group defined by $\mathbf{e}\in\mathcal{F}$ which is non-constant modulo~$p$, and such that for all $n$,
\[
\psi_1(G_{n-1}')\ge G_n'\times \overset{p^{m_n}}\dots\times G_n'.
\]
Then $G$ has the weak congruence subgroup property.
\end{lemma}

\begin{proof}
Let $H$ be a finite-index subgroup of~$G$. From the proof of \cite[Thm.~4]{NewHorizons}, it follows that $\psi_n^{-1}(G_n''\times \overset{p^{m_1+\cdots+m_n}}\dots\times G_n'')\le H$ for some $n$. Since $\psi_n(\st_G(n)')\le G_n'\times \overset{p^{m_1+\cdots+m_n}}\dots\times G_n'$,
from Lemma~\ref{lem:derived-stab-in-gamma-3}, it suffices to show that $\psi_1(G'')\ge \gamma_3(G_1)\times\overset{p^{m_1}}\dots\times \gamma_3(G_1)$. This follows from the fact that $\psi_1(G')\ge G_1'\times\overset{p^{m_1}}\dots\times G_1'$ together with the fact that $\psi_1(G')$ is a subdirect product of  $G_1\times\overset{p^{m_1}}\ldots\times G_1$.
\end{proof}

Thus we also obtain  the first examples of branch groups without the congruence subgroup property, but with the $p$-congruence subgroup property and the weak congruence subgroup property. Prior to this, the only examples of this kind were the weakly branch, but not branch, $p$-Basilica groups~\cite{DNT}.

To end this section, we remark that the assumption $\psi_1(G_{n-1}')\ge G_n'\times \overset{p^{m_n}}\dots\times G_n'$ (respectively $\psi_1(\gamma_3(G_{n-1}))\ge \gamma_3(G_n)\times \overset{p^{m_n}}\dots\times \gamma_3(G_n)$) for all $n$, in Lemmata~\ref{lem:derived-stab-in-gamma-3}, 
\ref{lem:stabdescription}, \ref{lem:weak-CSP},  Theorem~\ref{lem:branch-gamma-3}, and Corollary~\ref{cor:saturated} can be replaced by just the case for $n=1$. Indeed, this follows from the definition of $b_1$ together with Lemmata~\ref{lem:G'} and~\ref{lem:branching-subgroup}.

\subsection{Hausdorff dimension}\label{sec:Hdim}

In this final part of the section, we observe, analogous to \cite[Prop.~6.12]{Jorge}, that the Hausdorff dimension of the closure of a  growing GGS-group in $\Aut T$ is zero; recall that $T=T_{\overline{m}}$. Indeed, let $G$ be a growing GGS-group. The Hausdorff dimension of its closure $\overline{G}$ in $\Aut T$ is given by the formula
\[
\text{hdim}_{\Aut T}(\overline{G})=\varliminf_{k\to\infty}\frac{\log|G:\st_G(k)|}{\log |\Aut T:\st_{\Aut T}(k)|};
\]
compare with \cite{BaSh97}, where our filtration series is taken to be the series of level stabilisers $\st_{\Aut T}(k)$, for $k\in\mathbb{N}$.

As was done in the proof of \cite[Prop.~6.12]{Jorge}, we obtain
\begin{align*}
    \text{hdim}_{\Aut T}(\overline{G})&=\varliminf_{k\to\infty}\frac{\log|G:\st_G(k)|}{\log |\Aut T:\st_{\Aut T}(k)|}\\
    &= \varliminf_{k\to\infty}\frac{\log|G:\st_G(1)|+\log|\st_G(1):\st_G(2)|+\cdots+\log|\st_G(k-1):\st_G(k)|}{\log |\Aut T:\st_{\Aut T}(1)|+\cdots+\log |\st_{\Aut T}(k-1):\st_{\Aut T}(k)|}\\  
    &\le \varliminf_{k\to\infty}\frac{\log p^{m_1}+\log p^{m_1+m_2}+\cdots+\log p^{m_1+\cdots+m_k}}{\log |\st_{\Aut T}(k-1):\st_{\Aut T}(k)|}\\  
     &= \varliminf_{k\to\infty}\frac{\log \big(p^{m_1}p^{m_1+m_2}\cdots p^{m_1+\cdots+m_k}\big)}{\log \big(p^{m_1+\cdots+m_{k-1}}\cdot (p^{m_k}!)\big)}\\
     &=0,
\end{align*}
where for the third line we have used the fact that $\st_G(i-1)/\st_G(i)$ embeds into the direct product $C_{p^{m_i}}\times\overset{p^{m_1+\cdots+m_{i-1}}}\ldots\times C_{p^{m_i}}$.  

\smallskip


\section{Maximal subgroups}\label{sec:maximal}
In the present section we prove Theorem~\ref{start}. 
As is usually done,
we rephrase the argument in terms of proper prodense subgroups.

\subsection{Prodense subgroups}
We recall that a subgroup~$H$
of a group~$G$ is \textit{prodense} if $G=HN$ for all non-trivial normal subgroups~$N$ of~$G$; see~\cite[Def.~1.2]{GeGl}. By~\cite[Prop.~2.22]{Francoeur-paper}, for a finitely generated branch group, having maximal subgroups of infinite index is equivalent to having proper prodense subgroups. Using results on prodense subgroups from~\cite{Francoeur-paper}, we will prove Theorem~\ref{start} by showing that  those branch growing GGS-groups do not have proper prodense subgroups.

Let $G$ be a group acting on a spherically homogeneous rooted tree. For a vertex 
$u$ of the tree, we write 
\[
G_u = \varphi_u(\text{st}_G(u)).
\]

We require the following  results: 

\begin{lemma}\label{lem:prodense}\cite[Lem.~3.1]{Francoeur-paper}
Let $G$ be a  weakly branch group acting  on a spherically homogeneous rooted tree. Suppose $H\le G$ is a prodense subgroup and let $u$ be a  vertex of the tree. Then $H_u$ is a prodense subgroup of $G_u$.
\end{lemma}

\begin{theorem}\label{thm:proper}\cite[Thm.~3.2]{Francoeur-paper}
Let $G$ be a  weakly branch group acting on a spherically homogeneous rooted tree and let
$H\le G$ be a prodense subgroup.  Then $H\ne G$ if and only if $H_u\ne G_u$ for any vertex~$u$ of the tree.
\end{theorem}

Recall that $T_{\overline{m}}$ denotes the rooted tree with branching sequence $p^{\overline{m}}=(p^{m_1},p^{m_2},\ldots)$ where $m_1<m_2<\cdots$ are increasing positive integers. For the rest of the section, let
$G =\langle a,b\rangle$ be a branch growing GGS-group acting on $T=T_{\overline{m}}$. The strategy for proving Theorem~\ref{start} is to show that,  for $M$ a prodense subgroup of~$G$, there exists a vertex~$u$ at some level $i\in\mathbb{N}\cup \{0\}$ of~$T$ such that $a_i,b_i\in M_u$. We will do this in two steps: first by obtaining the $b$-generator, and then the $a$-generator.

\subsection{Obtaining $b$}

 For $\mathbf{e}=(e_1,\ldots,e_{p^{m_1}-1})$ the defining vector associated to~$G$, we recall the torsion-type condition~\eqref{eq:periodicity}:  
\[
\sum_{k=1}^{p^{m_1-i}-1} e_{kp^i} \equiv 0 \pmod {p^{i+1}}\quad\text{for all  }i\in\{0,\ldots,m_1-1\}.
\]

\begin{proposition}\label{proposition: first}
  Let $G= \langle a,b\rangle$ be a 
   branch growing GGS-group with defining vector satisfying the torsion-type condition, and let $M$ be a prodense  subgroup of~$G$.  Then there is a vertex $u$ of~$T$, of some level $i$, 
  such that  
  $b_i^{\,\delta}\in (M_u)^{a_i^{\,k}}$ for some $\delta\not\equiv 0 \pmod p$ and for some $k$. 
\end{proposition}

\begin{proof}
  Since $M$ is prodense and $G'$ is a non-trivial normal subgroup,  we can find some element $x
  \in M \cap b G'$.  In particular $x \in \st_G(1)$ with
  $\epsilon_{b}(x) =1\not\equiv 0\pmod p$.  We proceed by induction
  on $|x| \geq 1$.

  When $|x| = 1$, we have $x =
  b^{a^k}$.  
  Thus choosing the vertex $u$ to be the root of
  the tree $T$, we have $b \in (M_u)^{a^{-k}}$. 

  Now suppose that $|x| \ge 2$.  
 We claim that
  \begin{equation}\label{eq:1}
    \epsilon_{b_1}(\varphi_1 (x)) + \cdots + \epsilon_{b_1}(\varphi_{p^{m_1}}
    (x)) = \epsilon_{b} (x) =1. 
  \end{equation}
  Indeed, upon writing $x$ as a product of conjugates $b^{a^*}$,
  for some    unspecified exponents~$*$, we see that
  $\epsilon_{b} (x)$ is the  exponent sum  of all the
  $b^{a^*}$.  As each  $b^{a^*}$ gives~$b_1$ in a unique coordinate, the claim follows.

  By Equation~\eqref{eq:1}, there exists $j \in \{ 1,\ldots,p^{m_1} \}$ such that
  $\epsilon_{b_1} (\varphi_j(x))\not\equiv 0\pmod p$.  Moreover,
  Lemma~\ref{shortening} shows that $|\varphi_j(x)| 
  \le 
    (|x| + 1)/2  < |x|$.
  
  \medskip

  \underline{Case 1:} Let $\tilde x = \varphi_j(x)$ and suppose that $\tilde x \in M_{u(j)}$ belongs to
  $\st_{G_{u(j)}}(1)$; here we write $u(j)$ instead of $u_j$ for
  readability. By Lemma~\ref{lem:prodense}, the subgroup $M_{u(j)}$ is
  prodense in $G_{u(j)} = G_1$.  Since by our set-up $\epsilon_{b_1}(\tilde x) \not\equiv 0 \pmod p$ and $|\tilde x| < |x|$, the
  result follows by induction.

  \medskip

   \underline{Case 2:} Now suppose that $\varphi_j (x) \not \in \st_{G_{u(j)}}(1)$.  Write $\varphi_j (x) =a_1^{\,\alpha}h$ for $\alpha \not \equiv 0 \pmod{p^{m_2}}$ and $h \in \st_{G_{u(j)}}(1)$ with
  $\psi_1(h)= (h_1,\ldots, h_{p^{m_2}})$.
  
  \smallskip 
  
  \textit{Subcase (a):} Suppose $\alpha \not\equiv 0 \pmod{p}$.  For
  $\ell \in \{ 1, \ldots, p^{m_2}\}$ we claim that
  \begin{equation}\label{eq:2}
    \epsilon_{b_2}\big(\varphi_\ell (\varphi_j (x)^{p^{m_2}})\big) =
    \epsilon_{b_1} (\varphi_j (x)) \not\equiv 0\pmod p.
  \end{equation}
 Indeed,
  \[
  \varphi_j(x)^{p^{m_2}} = (a_1^{\,\alpha} h)^{p^{m_2}} = {h^{a_1^{\,(p^{m_2}-1)\alpha }}} {h^{a_1^{\,(p^{m_2}-2)\alpha }}}
  \cdots h^{a_1^{\,\alpha}} h,
  \]
 as $a_1^{\,p^{m_2}}=1$ and thus for any $\ell$,
  \[
  \varphi_\ell (\varphi_j (x)^{p^{m_2}}) \equiv h_1 h_2 \cdots h_{p^{m_2}}
  \pmod{G_{u(j\ell)}'}
  \]
  where $u(j\ell)=u_{j\ell}$ denotes the $\ell$th descendant of $u_j$.  Arguing
  similarly as for Equation~\eqref{eq:1}, this yields
  \[
  \epsilon_{b_2}(\varphi_\ell(\varphi_j(x)^{p^{m_2}})) =
  \epsilon_{b_2}(h_1)+\cdots +\epsilon_{b_2}(h_{p^{m_2}}) =
  \epsilon_{b_1}(h) =\epsilon_{b_1}(\varphi_j(x))
  \]
  and thus Equation~\eqref{eq:2} holds.

 Note that
  \begin{equation} \label{eq:got-shorter}
    |\varphi_\ell (\varphi_j(x)^{p^{m_2}})| \le |\varphi_j
    (x)| < |x|. 
  \end{equation}
  Indeed, Lemma~\ref{shortening} gives the second inequality. For the first inequality, we note that
  \[
  \varphi_\ell (\varphi_j (x)^{p^{m_2}}) = \varphi_\ell (h^{a_1^{\,(p^{m_2}-1)\alpha }})
  \cdots \varphi_\ell (h^{a_1^{\,\alpha}}) \varphi_\ell (h),
  \]
  and $|\varphi_j(x)| = |h|$.   We observe that $h=(b_1^*)^{a_1^*}\overset{|h|}\cdots(b_1^*)^{a_1^*}$, where $*$ denotes unspecified exponents. Each 
  $(b_1^*)^{a_1^*}$ gives~$b_2^*$ in a unique
  coordinate and powers of~$a_2$ in all other coordinates.  Therefore, we can write $\varphi_\ell (\varphi_j
  (x)^{p^{m_2}})$ as a product of powers of $a_2$ and the $|h|$ 
  automorphisms $b_2^*$. 
  Thus
 Equation~\eqref{eq:got-shorter} holds.

Since, due to the torsion-type condition,  the total exponent of~$a_2$ amongst the components of~$\psi_1(b_1)$ is a multiple of~$p$,  this implies that the total exponent of~$a_2$ in $ \varphi_\ell(\varphi_j(x)^{p^{m_2}})$ is a multiple of~$p$.
 
 If $\varepsilon_{a_2}(\varphi_\ell(\varphi_j(x)^{p^{m_2}}))=0$, then
 $\tilde x = \varphi_\ell(\varphi_j(x)^{p^{m_2}}) \in M_{u(j \ell)}$
  belongs to $\st_{G_{u(j \ell)}}(1)$.  As before, the subgroup $M_{u(j \ell)}$ is prodense in
  $G_{u(j \ell)} = G_2$.  Since $\epsilon_{b_2}(\tilde x) \not\equiv 0 \pmod p$ and $|\tilde x| < |x|$, the
  result follows by induction.

 If $\varepsilon_{a_2}(\varphi_\ell(\varphi_j(x)^{p^{m_2}}))\ne 0$, then we proceed as in the next subcase.
 
 \smallskip 
 
   \textit{Subcase (b):} Suppose $\varphi_j(x)=a_1^{\alpha}h$ with $\alpha \neq 0$ but $\alpha \equiv 0 \pmod{p}$. Then $\alpha=\lambda p^{d_2}$ for some $\lambda\not\equiv 0\pmod p$ and $1\le d_2<m_2$. We consider the components of $\psi_1\big(\varphi_j(x)^{p^{m_2-d_2}}\big)$.
   Certainly there exists some 
  $\ell$ such that $ \epsilon_{b_2}(\varphi_\ell (\varphi_j (x)^{p^{m_2-d_2}}))\not\equiv 0 \pmod p$. If $\tilde{x}=\varphi_\ell (\varphi_j (x)^{p^{m_2-d_2}})$ is in $\st_{G_2}(1)$ or yields an appropriate element in $\st_{G_3}(1)$ as seen in Case 2(a),  
  we proceed by induction, noting that $$
  |\varphi_\ell (\varphi_j (x)^{p^{m_2-d_2}})|\le| \varphi_j (x)|<|x|.
  $$
  Therefore suppose we are back in Case 2(b). 
 
 Without loss of generality,
 we may assume that $|\varphi_\ell (\varphi_j (x)^{p^{m_2-d_2}})|=| \varphi_j (x)|$. Indeed,  if $|\varphi_\ell (\varphi_j (x)^{p^{m_2-d_2}})|<| \varphi_j (x)|$, we simply replace $\varphi_j(x)$ with $\varphi_\ell (\varphi_j (x)^{p^{m_2-d_2}})$ and proceed as before. Now the fact that $|\varphi_\ell (\varphi_j (x)^{p^{m_2-d_2}})|=| \varphi_j (x)|$ means that, 
  writing 
  \[
  h= (b_1^*)^{a_1^*}\cdots(b_1^*)^{a_1^*}
  \]
  as a product of conjugates of~$b_1$, we have that  all exponents of~$a_1$ differ from one another by a multiple of~$p^{d_2}$. Let $\tilde{x}=\varphi_\ell (\varphi_j (x)^{p^{m_2-d_2}})= a_2^{\,\alpha_2}h_2$ for some $\alpha_2\ne 0$ and $h_2\in\st_{G_{u(j\ell)}}(1)$. From condition~\eqref{eq:periodicity}, it follows that   $p^{d_2+1}$ divides~$\alpha_2$.

  We now repeat the above process, that is, we let $1\le d_3 <m_3$ be the highest power of~$p$ that divides the total exponent of~$a_2$ in~$\tilde{x}$. As before, there exists an 
  $\ell_3 \in \{ 1,2,\ldots,  p^{m_3}\}$ such that $ \epsilon_{b_3}(\varphi_{\ell_3} (\tilde{x}^{p^{m_3-d_3}}))\not\equiv 0 \pmod p$. Likewise, if $\tilde{x}_3=\varphi_{\ell_3} (\tilde{x}^{p^{m_3-d_3}})$ is in $\st_{G_3}(1)$ or yields an appropriate element in $\st_{G_4}(1)$ as seen in Case 2(a),  
  we proceed by induction.
  So we suppose that we are again in Case~2(b) and that $|\tilde{x}_3|=|\tilde{x}|$. As the condition~\eqref{eq:periodicity} ensures that $\varepsilon_{a_3}(\tilde{x}_3)$ is divisible by $p^{d_2+2}$, and similarly for $\tilde{x}_4, \tilde{x}_5,\ldots$, if we are always in  Case 2(b) with $|\tilde{x}_n|=|\tilde{x}|$, then for some $n\ge 3$, we have $\tilde{x}_n\in\st_{G_n}(1)$ since $e_{k p^{m_1}}=0$ in the defining vector of~$b_n$, for $k\in\{1,\ldots, p^{m_n-m_1}-1\}$. Thus,  we may proceed as above; cf. Case~1.
\end{proof}

\subsection{Obtaining $a$}

Below we mimic the argument in~\cite{Pervova4} by adapting the definition of the $\Theta$ map there to the setting of growing GGS-groups.

For a growing GGS-group with defining vector $\mathbf{e}\in\mathcal{F}$  there exists an $n\in \{1,\ldots, p^{m_1}-1\}$ with $e_n\not\equiv 0 \pmod p$. 
Fix such an $n$ and let $d\in\mathbb{Z}^*$ be such that 
$de_n\equiv n \pmod {p^{m_2}}$. We note that the choice of~$d$ is not unique.

In the next subsection, we will be considering a prodense subgroup $M\le G$ that contains~$b^d$ 
and
an ``approximation" $a^nz \in a^n G'$ of $a^n$.
To obtain $a$ we aim to project, along the rightmost infinite path of the tree~$T$, appropriate elements from the first level
stabiliser $\st_M(1)$ 
to a subgroup of
$\Aut T_{v}$, for $v=p^{m_1}p^{m_2}\cdots p^{m_k}$ for some $k\in\mathbb{N}$.  
Writing $\psi_1(z) =(z_1,\ldots,z_{p^{m_1}})$ for an arbitrary element $z\in G'$ and
conjugating $b^d$ by $(a^nz)^{-1}$, by our choice of $d$ we obtain
\[
  \varphi_{p^{m_1}}\big((b^d)^{(a^nz)^{-1}}\big)=z_na_1^{\,n}z_n^{-1},
\]
equivalently,
\[
\varphi_{p^{m_1}}\big((b^d)^{(a^nz)^{-1}}\big) = a_1^{\,n}[a_1^{\,n},z_{n}^{-1}].
\]
We define
\[
\Theta\colon G' \rightarrow G_1', \quad \Theta(z)=[a_1^{\,n},z_{n}^{-1}].
\]
Since $\overline{m}$ is arbitrary, by abuse of notation we will still write $\Theta$ for the corresponding maps from $G_i'$ to $G_{i+1}'$ for each $i$. The integer~$d$ linked with the definition of $\Theta$ will also correspondingly change.

Let $S=\{a_1,b_1,a_1^{-1},b_1^{-1}\}$. As was done by Pervova~\cite{Pervova4}, for later use we define the following three types for an element~$z\in G'$. We say that $z\in G'$ is of 
\begin{enumerate}
    \item[$\bullet$] \emph{type one} if $|z_n|\ne \frac{|z|}{2}$,\;
    \item[$\bullet$] \emph{type two} if $|z_n|= \frac{|z|}{2}$ and $z_n$, written as a word in~$S$ of minimal length, begins with a non-trivial power of~$b_1$;
    \item[$\bullet$] \emph{type three} if $|z_n|= \frac{|z|}{2}$ and $z_n$, written as a word in~$S$ of minimal length, ends with a non-trivial power of~$b_1$. 
\end{enumerate}
Note that the above three types are the only possibilities for $z\in G'$.

\begin{lemma}\label{lem:2.3}
Let $G$ be a growing GGS-group with defining vector $\mathbf{e}\in \mathcal{F}$ and $n$ as defined above. Suppose that $z\in G'$ is of type three and let 
\[
z_n=a_1^{\,\alpha_1}b_1^{\, \beta_1}a_1^{\,\alpha_2}\cdots a_1^{\,\alpha_\ell}b_1^{\,\beta_\ell}
\]
where $\ell=\frac{|z|}{2}$, $\beta_1,\ldots, \beta_\ell\in \mathbb{Z}^*$, $\alpha_1\in \mathbb{Z}/p^{m_2}\mathbb{Z}$ and $\alpha_2,\ldots,\alpha_\ell\in (\mathbb{Z}/p^{m_2}\mathbb{Z})^*$. Then $\Theta(z)$ is of type three if and only if $\alpha_i=(-1)^in$ for  all $i\in \{1,2,\ldots,\ell\}$.
\end{lemma}

\begin{proof} This follows just as in~\cite[Lem.~2.3]{Pervova4}. 
\end{proof}

The following result is straightforward to verify, and first appeared  in~\cite[Cor.~2.4]{Pervova4} in the setting of GGS-groups acting on the $p$-regular rooted tree.

\begin{corollary}\label{cor:2.4}
Let $G$ be a growing GGS-group with defining vector $\mathbf{e}\in \mathcal{F}$ and $n$ as defined above. Suppose that both $z$ and $\Theta(z)$ are of type three. Writing $\ell=\frac{|z|}{2}$ and
\[
\Theta(z)=(b_1^{\,\alpha_1})^{a_1^{\,t_1}}(b_1^{\,\beta_1})^{a_1^{\,n}}(b_1^{\,\alpha_2})^{a_1^{\,t_2}}\cdots (b_1^{\,\beta_{\ell-1}})^{a_1^{\,n}}(b_1^{\,\alpha_\ell})^{a_1^{\,t_\ell}}(b_1^{\,\beta_\ell})^{a_1^{\,n}},
\]
where $\alpha_i,\beta_i\in \mathbb{Z}^*$ and $t_i\in  (\mathbb{Z}/p^{m_2}\mathbb{Z})^*$, we have
\begin{enumerate}
    \item if $\ell$ is even, then $t_i=\begin{cases}
    2n & \text{for }i\in \{1,2,\ldots,\frac{\ell}{2}\},\\
    0 & \text{for }i\in \{\frac{\ell}{2}+1,\ldots, \ell\},
    \end{cases}$
    \item if $\ell$ is odd, then $t_i=\begin{cases}
    2n & \text{for }i\in \{1,2,\ldots,\frac{\ell+1}{2}\},\\
    0 & \text{for }i\in \{\frac{\ell+3}{2},\ldots, \ell\}.
    \end{cases}$
\end{enumerate}
\end{corollary}

The following result is proved as in~\cite[Thm.~2.1]{Pervova4}, via a careful case-by-case analysis. However, owing to the fact that $\varphi_i(b_1)=1$ for $i\in \{p^{m_1},p^{m_1}+1,\ldots,p^{m_2}-1\}$, many complicated cases in Pervova's proof disappear.

\begin{theorem}\label{thm:obtaining-a}
Let $G$ be a growing GGS-group with defining vector $\mathbf{e}\in \mathcal{F}$ and $n$ as defined above. If $K=\langle b^{\delta}, a^nz\rangle$ for  some $\delta\not\equiv 0 \pmod p$ and $z\in G'$, then there exists some $k\in\mathbb{N}$ such that, writing $v=p^{m_1}p^{m_2}\cdots p^{m_k}$, we have $\varphi_v(\textup{st}_K(v))\ge\langle a_k, b_k^{\,\delta}\rangle^{b_k^{\,\lambda}}$ for some 
$\lambda\in \mathbb{Z}$.
\end{theorem}

\begin{proof}
It suffices to show that there is some $k\in\mathbb{N}$ such that $\varphi_v(\text{st}_K(v))\ge\langle a_k^{\,n}, b_k^{\,\delta}\rangle^{b_k^{\,\lambda}}$ for some $\lambda\in \mathbb{Z}$ and where $v=p^{m_1}p^{m_2}\cdots p^{m_k}$, as then upon considering
\[
\varphi_{p^{m_{k+1}}}\big((b_k^{\,\delta})^{a_k^{\,-n}}\big)=a_{k+1}^{\,\mu}
\quad\text{and}\quad \varphi_{p^{m_{k+1}}}(b_k^{\,\delta})=b_{k+1}^{\,\delta}, 
\]
where $\mu\not\equiv 0 \pmod p$, the result follows.

Let $\varepsilon\in\mathbb{Z}^*$ be such that $\varphi_n(b^{\delta\varepsilon})=a_1^{\,n}$, and we let $d=\delta\varepsilon$ for the definition of~$\Theta$. Note that 
\begin{equation}\label{eq:theta-reduction}
|\Theta(z)|\le 2|z_n|.
\end{equation}
We will utilise this to proceed by induction on the length~$|z|$.

For the base step, we suppose that $|z|\le 2$. As $z\in G'$, either $|z|=0$ or $|z|=2$. If $|z|=0$, then $z=1$ and we are done. Hence suppose that $|z|=2$. Then $z=(b^r)^{a^i}(b^{-r})^{a^j}$ for some $r\in \mathbb{Z}^*$ and $i,j\in\mathbb{Z}/p^{m_1}\mathbb{Z}$ with $i\ne j$. There are three possibilities to consider.

\underline{Case 1:} Suppose $i,j\ne n$. Then $z_n=a_1^{\,s}$ for some $s\in\mathbb{Z}/p^{m_2}\mathbb{Z}$. Hence $\Theta(z)=1$. Since there exists an element of $K$ whose image under $\varphi_{p^{m_1}}$ is $a_1^{\,n}\Theta(z)$ (by the definition of the $\Theta$ map above), we conclude that  $\varphi_{p^{m_1}}(\st_K(1))=\langle a_1^{\,n}, b_1^{\,\delta}\rangle$.

\underline{Case 2:} Suppose $i=n$. Then $j\ne n$ and $z_n=b_1^{\,r} a_1^{\,-r e_{n-j}}$. Recalling the choice of $d$ above, we have 
\[
\varphi_{p^{m_1}}\big((b^d)^{(a^nz)^{-1}}\big)=b_1^{\, r}a_1^{\, n}b_1^{\, -r},
\]
and hence $\varphi_{p^{m_1}}(\st_K(1))\ge \langle b_1^{\,\delta},b_1^{\, r}a_1^{\, n}b_1^{\, -r}\rangle=\langle a_1^{\, n}, b_1^{\,\delta}\rangle^{b_1^{\, -r}}$.

\underline{Case 3:} Suppose $j=n$. Then $i\ne n$ and $z_n=a_1^{\, re_{n-i}} b_1^{\,-r}$, giving
\[
\Theta(z)=(b_1^{\, -r})^{a_1^{\, n-re_{n-i}}}(b_1^{\,r})^{a_1^{\,-re_{n-i}}}.
\]
If neither $n-re_{n-i}$ nor $-re_{n-i}$ is congruent  to~$n$ modulo~$p^{m_2}$, then we proceed as in Case~1. If $n-re_{n-i}\equiv n \pmod {p^{m_2}}$, then we proceed as in Case~2. Hence we assume that $-re_{n-i}\equiv n \pmod {p^{m_2}}$.

Here we have $\Theta(z)=(b_1^{\, -r})^{a_1^{\,2n}} (b_1^{\, r})^{a_1^{\, n}} $. Recalling how $b_1$ is defined, in particular $e_\ell=0$ for all $p^{m_1}\le \ell\le p^{m_2}-1$, we see that $\varphi_n(\Theta(z))= a_2^{\, -re_{p^{m_2}-n}} b_2^{\,r}=b_2^{\, r}$. Therefore,
\[
\Theta(\Theta(z))=(b_2^{\, r})^{a_2^{\, n}}b_2^{\, -r}.
\]
Writing $u=p^{m_1}p^{m_2}$, as $\varphi_u(\text{st}_K(u))\ge \langle a_2^{\,n}\Theta(\Theta(z)),b_2^{\,\delta}\rangle$ we proceed as in Case~2.

\medskip

Next, for the induction step, we also have three cases to consider, depending on the type of~$z$. 

First suppose that $z$ is of type one. Then either $|z_n|<\frac{|z|}{2}$ or $|z_n|>\frac{|z|}{2}$.
If $|z_n|<\frac{|z|}{2}$, it follows from Equation~\eqref{eq:theta-reduction} that
\[
|\Theta(z)|<|z|,
\]
and we can proceed by induction. Hence we suppose that  $|z_n|>\frac{|z|}{2}$. From Lemma~\ref{shortening}, it follows that $|z|$ is odd and that $|z_n|=\frac{|z|+1}{2}$. Writing $n=p^s t$, for $s\in\{0,1,\ldots,m_1-1\}$ and $t\not\equiv 0 \pmod p$, observe that 
\begin{align*}
\varphi_{p^{m_1}} \big((a^nz)^{p^{m_1-s}}\big)&=\varphi_{p^{m_1}} \big(z^{a^{tp^{m_1}-n}}z^{a^{tp^{m_1}-2n}}\cdots z^{a^{2n}}z^{a^{n}}z\big)\\
&=\varphi_{p^{m_1}} \big(z^{a^{n(p^{m_1-s}-1)}}z^{a^{n(p^{m_1-s}-2)}}\cdots z^{a^{2n}}z^{a^{n}}z\big)\\
&= z_{n}z_{2n}\cdots z_{p^{m_1}-n} z_{p^{m_1}},
\end{align*}
where the indices~$\ell$ of $z_\ell$ are reduced modulo $p^{m_1}$. 
Therefore
\begin{align*}
\varphi_{p^{m_1}} \big((b^d)^{(a^nz)^{p^{m_1-s}-1}}\big)&=(z_na_1^{\,n}z_n^{-1})^{z_{n}z_{2n}\cdots z_{p^{m_1}-n} z_{p^{m_1}}}\\
&=z_{p^{m_1}}^{-1} z_{p^{m_1}-n}^{-1}\cdots z_{2n}^{-1} a_1^{\, n} z_{2n}\cdots z_{p^{m_1}-n} z_{p^{m_1}}.
\end{align*}
As 
\[
|z_{2n}\cdots z_{p^{m_1}-n} z_{p^{m_1}}|\le \sum_{
i\ne n} |z_i| \le |z|-|z_n|,
\]
we obtain
\[
\big\vert a_1^{\,-n}\varphi_{p^{m_1}} \big((b^d)^{(a^nz)^{p^{m_1-s}-1}}\big)\big\vert\le 2|z|-2|z_n|=2|z|-|z|-1<|z|
\]
with of course $a_1^{\,-n}\varphi_{p^{m_1}} \big((b^d)^{(a^nz)^{p^{m_1-s}-1}}\big)\in G_1'$.
Therefore, since
\[
\varphi_{p^{m_1}}(\st_K(1))\ge \Big\langle b_1^{\,\delta}, a_1^{\,n} \big(a_1^{\,-n}\varphi_{p^{m_1}} \big((b^d)^{(a^nz)^{p^{m_1-s}-1}}\big)\big)\Big\rangle,
\]
the result follows by induction.

Now we suppose that $z$ is of type two. Then $z_n$ has the form $z_n=b_1^{\,\beta_1} a_1^{\, \alpha_1}\cdots b_1^{\, \beta_\ell}a_1^{\, \alpha_\ell}$ for some $\beta_1,\ldots, \beta_{\ell}\in \mathbb{Z}^*$, $\alpha_1,\ldots, \alpha_{\ell-1}\in (\mathbb{Z}/p^{m_2}\mathbb{Z})^*$  and $\alpha_\ell\in \mathbb{Z}/p^{m_2}\mathbb{Z}$ where $\ell=\frac{|z|}{2}$. Thus
\[
\Theta(z)=a_1^{\,-n}b_1^{\,\beta_1}a_1^{\,\alpha_1}b_1^{\,\beta_2}\cdots a_1^{\,\alpha_{\ell-1}}b_1^{\, \beta_\ell} a_1^{\, n} b_1^{\,-\beta_\ell}a_1^{\, -\alpha_{\ell-1}}\cdots b_1^{\,- \beta_2}a_1^{\, -\alpha_1}b_1^{\,-\beta_1}
\]
and $|\Theta(z)|\leq 2|z_n|=|z|$.
If $|\Theta(z)|<|z|$, we proceed as before by induction. So suppose $|\Theta(z)|=|z|$. We deduce that $\alpha_i=(-1)^in$ for all $i\in\{1,\ldots,\ell\}$ (else  $|\varphi_n(\Theta(z))|<|z_n|$), which gives
\[
\Theta(z)=\begin{cases}
(b_1^{\,\beta_1})^{a_1^{\,n}}(b_1^{\,\beta_2})^{a_1^{\,2n}}\cdots (b_1^{\,\beta_{\ell-1}})^{a_1^{\,n}}(b_1^{\,\beta_\ell})^{a_1^{\,2n}} (b_1^{\,-\beta_\ell})^{a_1^{\,n}}b_1^{\,-\beta_{\ell-1}}\cdots (b_1^{\, -\beta_2})^{a_1^{\, n}}b_1^{\,-\beta_1} & \text{ if }\ell\text{ is even,}\\
(b_1^{\,\beta_1})^{a_1^{\,n}}(b_1^{\,\beta_2})^{a_1^{\,2n}}\cdots (b_1^{\,\beta_{\ell}})^{a_1^{\,n}} b_1^{\,-\beta_\ell}(b_1^{\,-\beta_{\ell-1}})^{a_1^{\,n}}\cdots (b_1^{\,- \beta_2})^{a_1^{\, n}}b_1^{\,-\beta_1} & \text{ if }\ell\text{ is odd.}
\end{cases}
\]
From the definition of~$b_1$, it follows that \[
\varphi_n(\Theta(z))=\begin{cases}
b_2^{\, \beta_1+\beta_3+\cdots + \beta_{\ell-1}-\beta_\ell}a_2^{\,-e_n\beta_{\ell-1}}b_2^{\,-\beta_{\ell-2}}\cdots  a_2^{\,-e_n\beta_3}b_2^{\, -\beta_2} a_2^{\,-e_n\beta_1} & \text{ if }\ell\text{ is even,}\\
b_2^{\, \beta_1+\beta_3+\cdots + \beta_{\ell}}a_2^{\,-e_n\beta_{\ell}}b_2^{\,-\beta_{\ell-1}}\cdots  a_2^{\,-e_n\beta_3}b_2^{\, -\beta_2} a_2^{\,-e_n\beta_1} & \text{ if }\ell\text{ is odd,}
\end{cases}
\]
which is of length at most~$\frac{\ell+1}{2}$. 
Hence $|\Theta(\Theta(z))|<|\Theta(z)|$. 
As before, we then proceed by induction.

Finally, we suppose that $z$ is of type three. So $z_n= a_1^{\, \alpha_1}b_1^{\, \beta_1}\cdots a_1^{\, \alpha_\ell}b_1^{\, \beta_\ell}$ for some $\beta_1,\ldots, \beta_\ell\in \mathbb{Z}^*$, $\alpha_2,\ldots, \alpha_\ell\in (\mathbb{Z}/p^{m_2}\mathbb{Z})^*$ and $\alpha_1\in \mathbb{Z}/p^{m_2}\mathbb{Z}$.  

First we consider the case $\ell\ge 3$.
We may suppose that $\Theta(z)$ is of type three, else we are done by the above.   Let us suppose that $\ell$ is even; the case $\ell$ is odd follows similarly. Corollary~\ref{cor:2.4} implies that 
\[
\Theta(z)=(b_1^{\, \beta_1})^{a_1^{\, 2n}}(b_1^{\, \beta_2})^{a_1^{\, n}} (b_1^{\, \beta_3})^{a_1^{\, 2n}} \cdots (b_1^{\, -\beta_3})^{a_1^{\, n}}  b_1^{\, -\beta_2} (b_1^{\, -\beta_1})^{a_1^{\, n}}.
\]
As seen above, we obtain $|\Theta(\Theta(z))|<|\Theta(z)|$ and the result follows by induction.

It remains to settle the case $\ell=2$. Here we have $z_n=a_1^{\, \alpha_1} b_1^{\, \beta_1} a_1^{\, \alpha_2}b_1^{\, \beta_2}$. As before, if $\Theta(z)$ is of type one or two, we are done as above. Thus we assume that $\Theta(z)$ is of type three, and 
Corollary~\ref{cor:2.4} yields that 
\[
\Theta(z)=(b_1^{\, \beta_1})^{a_1^{\, 2n}}(b_1^{\, \beta_2})^{a_1^{\, n}} b_1^{\, -\beta_2} (b_1^{\, -\beta_1})^{a_1^{\, n}}.
\]
Recalling that $\varphi_i(b_1)=1$ for $i\in \{p^{m_1},p^{m_1}+1,\ldots,p^{m_2}-1\}$, it follows that $\varphi_n(\Theta(z))= b_2^{\, \beta_2}a_2^{\,-e_n\beta_2}b_2^{\, \beta_1}$.
Hence $\Theta(\Theta(z))=a_2^{\,-n} b_2^{\, \beta_2}a_2^{\,-e_n\beta_2 }b_2^{\, \beta_1} a_2^{\, n}
b_2^{\, -\beta_1}a_2^{\,e_n\beta_2}b_2^{\, -\beta_2}$, and 
therefore $\Theta(\Theta(z))$ is of type one or two. In either case, we are done as above.
\end{proof}


\subsection{Maximal subgroups of finite index}
Equipped with the previous two subsections, we can now prove Theorem~\ref{start}.

\begin{proposition}\label{proposition: second}
  Let $G$ be a  branch growing GGS-group with defining vector $\mathbf{e}\in\mathcal{F}$.
  Let $M$ be a prodense subgroup of~$G$, and suppose that $b^{\delta}\in M$ for some $\delta\not\equiv 0 \pmod p$. Then there
  exists a vertex~$u$ of~$T$, of some level $i\in\mathbb{N}\cup\{0\}$, such that $M_u=G_i$.
\end{proposition}

\begin{proof}
   As before,
  there is $z \in G'$ such that $a^nz \in M$, where $n$ is as defined in the previous subsection.  From Theorem~\ref{thm:obtaining-a}, we have that $M_u\ge\langle a_k, b_k^{\,\delta}\rangle^{b_k^{\,\lambda}}$, for some $k\in\mathbb{N}$, $\lambda\in\mathbb{Z}$, and vertex~$u$ of level~$k$.  We note that for a vertex $v$ of $T$  and $g\in G_v$, 
\begin{equation}\label{eq:conjugateaway}
  (M_{v})^g = G_{v} \quad\Longleftrightarrow \quad M_v = G_v.
\end{equation}
Hence, without loss of generality we may suppose that $M_u\ge\langle a_k, b_k^{\,\delta}\rangle$. Since $\overline{m}$ is arbitrary, we may also assume that $u$ is the root of the tree. 

Now for every vertex~$v$ of the tree~$T$, observe that
  \begin{equation}
      \label{eq:a-everywhere}
 M_{v}\ge\langle a_{\ell}, b_{\ell}^{\,\delta}\rangle,
  \end{equation}
  where $\ell$ is the level of~$v$ in~$T$. 
  We now proceed as in the proof of Proposition~\ref{proposition: first}. That is, as $M$ is prodense, we find $x\in M\cap b G'$. We consider a component $\varphi_v(x)$, for some $v\in X_1$, such that $x_v:=\varphi_v(x)$ satisfies $\varepsilon_{b_1}(x_v)\ne 0$. By Equation~\eqref{eq:a-everywhere}, we may assume that $x_v\in \st_{G_1}(1)$.  Repeating this process, noting that $|x_v|<|x|$ when $|x|>1$, we eventually end up with an element of the form $b_{k_1}^{\,d_1}$ for some $k_1\in\mathbb{N}$ and $d_1\in\mathbb{Z}^*$. We terminate this process as soon as we reach such a directed automorphism. In this recursive process, we use the notation $x_{w}:=\varphi_u(x_{v})$ for $u\in X_2$, where $w=vu$, and similarly for other levels. 
  Upon doing this process for every
  section $x_\lambda$ of $x$ with $|x_\lambda|\ge 1$ and $\varepsilon_{b_1}(x_\lambda)\neq 0$, we obtain a collection of directed automorphisms $\big\{b_{k_i}^{\,d_i} \mid i\in I\big\}$ for some finite set~$I$, distributed at various vertices across the tree. Writing $k=\max\{k_i\mid i\in I\}$, by Equation~\eqref{eq:a-everywhere} we have $b_k^{\,d_i}\in M_\nu$ for all $i\in I$, where $\nu$ is any vertex of level~$k$.
  Noting that
  \[
  1=\varepsilon_b(x)=\sum_{i\in I} d_i,
  \]
the product of all the $b_k^{\,d_i}$ gives us $b_k$, as required.
\end{proof}

\begin{proposition} \label{isG} Let $G$ be a  branch growing GGS-group with defining vector satisfying the torsion-type condition.  Let $M$ be a prodense subgroup of~$G$. Then there exists a
  vertex~$u$ of~$T$ such that $M_u = G_u$.
\end{proposition}

\begin{proof}
By Proposition~\ref{proposition: first}, there exists a vertex~$u_1$ of $T$ of level~$j$ and an element $g\in G_{u_1}$ such that $b_j^{\,\delta} \in (M_{u_1})^g$ for some $\delta\not\equiv 0 \pmod p$.
 Using Equation~\eqref{eq:conjugateaway} we may assume without loss
of generality that $b_j^{\,\delta} \in M_{u_1}$.
Then by Proposition~\ref{proposition: second}, there exists a vertex~$v$
of~$T_{u_1}$ of level~$i-j$ 
with $G_i=
M_{u_1v}$. As $G_{u_1v}=G_i$, the result follows.
\end{proof}

\begin{proof}[Proof of Theorem~\ref{start}]
  Suppose on the contrary that $M$ is a proper prodense  subgroup of $G$. By Theorem~\ref{thm:proper},
  for every vertex $u\in T$  we have $M_u$ is properly contained in
  $G_u$. However, by Proposition~\ref{isG}, there is a vertex $v$ of~$T$ such that the subgroup~$M_v$ is all
  of~$G_v$. This gives the required contradiction.
\end{proof}



\subsection{Weakly maximal subgroups}\label{sec:weakly-maximal}

The following results were proved in~\cite{BR}  for regular branch groups, but the proofs extend naturally to all branch groups. 

\begin{theorem}\cite[Analogue of Thm.~1.4]{BR}
Let $T$ be a spherically homogeneous rooted tree and $G\le \textup{Aut }T$ be a finitely generated branch group. 
Suppose 
that   there exists an integer~$\ell$ such that, for each $i\in \mathbb{N}\cup\{0\}$, there is an infinite-index subgroup $Q_i$ 
 of $\varphi_u(\textup{st}_G(u))$, for $u$ any $i$th-level vertex, that does not stabilise any vertex of the $\ell$th layer.
Then 
for any $k\in \mathbb{N}$ there exists a weakly maximal subgroup of $G$ which stabilises the $k$th layer and does not stabilise any vertex of the $(k+\ell)$th layer.
\end{theorem}

In the above result, 
if $G$ is  a branch growing GGS-group, then we can take $\ell=1$ and $Q_i=\langle a_i\rangle$.

\begin{theorem}\cite[Analogue of Thm.~1.1]{BR}
Let $T$ be a spherically homogeneous rooted tree and $G\le \textup{Aut }T$ be a finitely generated branch group. Then, for any finite subgroup $Q\le G$ there exist uncountably many automorphism equivalence classes of weakly maximal subgroups of~$G$ containing~$Q$.
\end{theorem}

\begin{corollary}\cite[Analogue of Cor.~1.2]{BR}
Let $T$ be a spherically homogeneous rooted tree and $G\le \textup{Aut }T$ be a finitely generated branch group. Suppose that $G$ contains a finite subgroup $Q$ that does not fix any point in~$\partial T$. Then there exist uncountably many automorphism equivalence classes of weakly maximal subgroups of~$G$, all distinct from classes of parabolic subgroups associated with the action of~$G$ on~$T$.
\end{corollary}

Recall that stabilizers of boundary points are called \emph{parabolic subgroups}. Therefore, for a branch growing GGS-group~$G$, taking $Q=\langle a\rangle$, it follows that there are uncountably many automorphism equivalence classes of weakly maximal subgroups of~$G$, all distinct from classes of parabolic subgroups associated with the action of~$G$ on~$T$.


\section{Quotients of growing GGS-groups}\label{sec:Beauville}

Finally, we prove Proposition~\ref{Beauville}. 
We recall that the condition for a finite group~$H$ to be a Beauville group can be reformulated in group-theoretical terms as follows: for $x,y\in H$, let
 \[
 \Sigma(x,y)=\bigcup_{g\in G} \big(\langle x\rangle^g \cup \langle y\rangle^g \cup \langle xy\rangle^g\big),
 \]
that is, the union of all conjugates of the cyclic subgroups generated by $x$, $y$ and $xy$. Then $H$ is a Beauville group if and only if $H$ is 2-generated and there exist generating sets $\{x_1,y_1\}$ and $\{x_2,y_2\}$ of~$H$ such that $\Sigma(x_1,y_1)\cap \Sigma(x_2,y_2)=\{1\}$.  
The sets $\{x_1,y_1\}$ and $\{x_2,y_2\}$ are then called a \emph{Beauville structure} for~$H$.

First we need some preliminary results, and we introduce some notation. Recall that  for a vertex~$w$ at level~$n-1$, 
the map~$\varphi_w$ takes an element in the stabiliser of the vertex~$w$ and projects onto the $w$-th section. Then let $\psi_n^{w}=\psi_1\circ\varphi_{w}$ which has $ \varphi_{w}^{-1}\big(\varphi_{w}(\text{st}_G(w))\cap \st_{G_{n-1}}(1)\big)$ as its domain and $G_n\times \overset{p^{m_n}}\ldots\times G_n$ as its codomain.

\begin{lemma}\label{b-order}
Let $G=\langle a,b\rangle$ be a growing GGS-group with defining vector $\mathbf{e}\in \mathcal{F}$. Then for each $n\ge 2$ the order of $b$ in $G/\st_G(n)$ is $p^{m_n}$.
\end{lemma}

\begin{proof}
For all $n\ge 2$, observe that as $a_k^{\,p^{m_{n-1}}}=1$ for all $1\le k\le n-2$,
\begin{align*}
&\psi_{n-1}^w(b^{p^{m_{n-1}}})\\
&\quad=\begin{cases}
(1,\overset{p^{m_{n-1}}}\ldots, 1) & \text{if }w\ne\prod_{k=1}^{n-2} p^{m_k},\\
(a_{n-1}^{\,(e_1)p^{m_{n-1}}},a_{n-1}^{\,(e_2)p^{m_{n-1}}},\ldots,a_{n-1}^{\,(e_{p^{m_1}-1})p^{m_{n-1}}},1,\overset{p^{m_{n-1}} -p^{m_1}}\ldots,1, b_{n-1}^{\,p^{m_{n-1}}}) & \text{if }w=\prod_{k=1}^{n-2} p^{m_k}.
\end{cases}
\end{align*}
Since the order of~$a_{n-1}$ is $p^{m_n}$ and recalling that $\mathbf{e}\in\mathcal{F}$, it follows that
\[
\psi_{n-1}(b^{p^{m_n}})=(1,\overset{p^{(m_1+\cdots +m_{n-1})}-1}\ldots,1,b_{n-1}^{\,p^{m_n}}),
\]
and hence the order of $b$ in $G/\st_G(n)$ is $p^{m_n}$, as required.
\end{proof}

\begin{lemma}\label{lem:order}
Let $G$ be a growing GGS-group associated to a defining vector $\mathbf{e}\in \mathcal{F}$ of zero sum and let $n\ge 3$.  Then for $1\le i \le p-1$, $0\le s\le m_2-1$, the order of $ab^{ip^s}$ is $p^{m_n+m_1-s}$ in $G/\st_G(n)$.
\end{lemma}

\begin{proof}
Let $1\le  i\le p-1$ and $0\le s \le m_2-1$. Now 
$$
(ab^{ip^s})^{p^{m_1}} = (b^{ip^s})^{a^{{p^{m_1} }-1}}(b^{ip^s})^{a^{{p^{m_1} }-2}}\cdots (b^{ip^s})^{a}b^{ip^s},
$$
and hence, as $e_1+\cdots+e_{p^{m_1}-1}=0$,
\[
\psi_1\big((ab^{ip^s})^{p^{m_1}}\big) = \big((b_1^{\,ip^s})^{a_1^*},\ldots ,(b_1^{\,ip^s})^{a_1^*}, b_1^{\,ip^s},b_1^{\,ip^s} \big),
\]
where $*$ stands for unspecified exponents. Then
\[
\psi_1\big((ab^{ip^s})^{p^{m_n+m_1-s}}\big) = \big((b_1^{ip^{m_n}})^{a_1^*},\ldots ,(b_1^{ip^{m_n}})^{a_1^*}, b_1^{ip^{m_n}}, b_1^{ip^{m_n}}\big).
\]
Thus, upon applying Lemma~\ref{b-order} to~$G_1$, we see that the order of $ab^{ip^s}$ is $p^{m_n+m_1-s}$ in $G/\st_G(n)$.
\end{proof}

The proof of the next result follows just as in~\cite[Prop.~3.4]{GUA}, but due to the more complicated notation, we give a full proof for completeness. In the following, by abuse of notation, we still write $\psi_i$, $i\in\mathbb{N}$, for the corresponding map applied to the quotient $G/\st_G(n)$.

\begin{lemma}\label{key}
Let $G$ be a growing GGS-group associated to a defining vector $\mathbf{e}\in \mathcal{F}$ of zero sum and let $n\ge 3$. In the quotient group $G/\st_G(n)$, if
\[
\langle (ab^{ip^s})^{p^{m_n+m_1{-s}-1}}\rangle = \langle (ab^{jp^t})^{p^{m_n+m_1{-t}-1}}\rangle^g
\]
for $1\le i,j\le p^{m_n}-1$ with $i,j\not\equiv 0 \pmod p$, $0\le s,t\le m_2-1$, and $g\in G/\st_G(n)$, then $ip^s\equiv jp^t \pmod {p^{m_2}}$.
\end{lemma}

\begin{proof}
 Write $w_{\ell,r}=(ab^{\ell p^r})^{p^{m_n+m_1{-r}-1 }}$ for  $1\le \ell\le p^{m_n}-1$ with $\ell\not\equiv 0\pmod p$, and $0\le r\le m_2-1$. Note that by Lemma~\ref{lem:order}, the groups $\langle w_{i,s}\rangle$ and  $\langle w_{j,t}\rangle^g$ are both of order~$p$.
  As $\langle w_{i,s}\rangle = \langle w_{j,t}\rangle^g$ we obtain
 \begin{equation}\label{eq:w-intersect}
 w_{i,s}^k=w_{j,t}^g 
 \end{equation}
 for some $1\le k\le p-1$. Now
 \begin{align*}
 \psi_1(w_{i,s})=\big((b_1^{\,ip^{m_n-1}})^{a_1^{ip^se_1}},(b_1^{\,ip^{m_n-1}})^{a_1^{ip^s(e_1+e_2)}},\ldots,
 (b_1^{\,ip^{m_n-1}})^{a_1^{ip^s(e_1+\cdots +e_{p^{m_1}-2})}} ,b_1^{\,ip^{m_n-1}},b_1^{\,ip^{m_n-1}}\big).
 \end{align*}
Likewise, the components of $\psi_1(w_{j,t}^g)$ are of the form $(b_1^{\,jp^{m_n-1}})^{a_1^{\,\mu}h}$, where $\mu\in\{0,\ldots,p^{m_2}-1\}$ and $h\in \st_{G_1}(1)/\st_{G_1}(n-1)$. 
 
 As one of the components of $\psi_1(w_{i,s}^k)$ is $b_1^{\,ikp^{m_n-1}}$, it follows from Equation~\eqref{eq:w-intersect} that $(b_1^{\,jp^{m_n-1}})^{a_1^{\,\mu}h}=b_1^{\,ikp^{m_n-1}}$ in $G_1/\st_{G_1}(n-1)$ for some~$\mu$ and~$h$. Therefore, 
 \[
 b_1^{\,(ik-j)p^{m_n-1}}=[b_1^{\,jp^{m_n-1}},a_1^{\,\mu}h]\in (G_1/\st_{G_1}(n-1))',
 \]
 and hence $ik-j\equiv 0 \pmod p$. Thus as the groups generated by $w_{i,s}$ and $w_{j,t}^g$ both have order~$p$ we have that $w_{i,s}^{i^{-1}j}=w_{j,t}^g$, where $i^{-1}\in\mathbb{F}_p^*$. Writing $x_{\ell,r}=w_{\ell,r}^{\ell^{-1}}$ 
 gives
 \[
 x_{i,s}=x_{j,t}^g.
 \]
 
 Note that
 \[
 \psi_1(x_{i,s})=\big((b_1^{\,p^{m_n-1}})^{a_1^{\,ip^se_1}},(b_1^{\,p^{m_n-1}})^{a_1^{\,ip^s(e_1+e_2)}},\ldots , b_1^{\,p^{m_n-1}}  ,b_1^{\,p^{m_n-1}}\big).
 \]
  Write $g=a^{q}h_q$ for some $0\le q\le p^{m_1}-1$ and $h_q\in \st_G(1)/\st_G(n)$. Observe that
 \begin{align*}
 \psi_1(x_{j,t}^{a^q})=\big((b_1^{\,p^{m_n-1}})^{a_1^{\,jp^t(e_1+\cdots +e_{p^{m_1}-(q-1)})}},\ldots, b_1^{\,p^{m_n-1}}, 
 b_1^{\,p^{m_n-1}},
 \ldots, (b_1^{\,p^{m_n-1}})^{a_1^{\,jp^t(e_1+\cdots+e_{p^{m_1}-q})}}\big),
 \end{align*}
 where the first $b_1^{\,p^{m_n-1}}$ occurs at the $(q-1)$st component. The equation $x_{i,s}=x_{j,t}^g$ now implies that
 \begin{align*}
 \psi_1(h_q)=\big(a_1^{\,ip^se_1-jp^t(e_1+\cdots+e_{p^{m_1}-(q-1)})}u_1&, a_1^{\,ip^s(e_1+e_2)-jp^t(e_1+\cdots+e_{p^{m_1}-(q-2)})}u_2, \ldots, \\
 &\qquad \qquad \qquad\quad \qquad a_1^{-jp^t(e_1+\cdots+e_{p^{m_1}-q})}u_{p^{m_1}}\big),
 \end{align*}
 where $u_\ell \in \st_{G_1}(1)/\st_{G_1}(n-1)$ for all $1\le \ell \le p^{m_1}$.

If $q>0$, we next recursively define  elements $h_{\ell-1}=h_{\ell}(b^{-j})^{a^{\ell-1}}$ for  $\ell=q,\ldots,1$. Since $e_1+\cdots+e_{p^{m_1}-1}=0$, we have
 \[
 \psi_1(h_0)=\big(a_1^{\,(ip^s-jp^t)e_1}v_1, a_1^{\,(ip^s-jp^t)(e_1+e_2)}v_2, \ldots,a_1^{\,(ip^s-jp^t)(e_1+\cdots+e_{p^{m_1}-2})}v_{p^{m_1}-2},v_{p^{m_1}-1},v_{p^{m_1} }\big)
 \]
 with $v_\ell\in\st_{G_1}(1)/\st_{G_1}(n-1)$ for all $1\le \ell \le p^{m_1}$.
 
 Observe that
 \[
 \psi_1(h_0^{\,a})\equiv \psi_1(h_0b^{jp^t-ip^s})\quad \pmod {\tfrac{\st_{G_1}(1)}{\st_{G_1}(n-1)}\times \overset{p^{m_1}}\ldots\times \tfrac{\st_{G_1}(1)}{\st_{G_1}(n-1)}}.
 \]
 Hence
 \[
 \psi_1(b^{ip^s-jp^t}[h_0,a])\in \left(\tfrac{\st_{G_1}(1)}{\st_{G_1}(n-1)}\times \overset{p^{m_1}}\cdots\times \tfrac{\st_{G_1}(1)}{\st_{G_1}(n-1)}\right)\,\, \bigcap\,\, \psi_1\Big(\tfrac{\st_{G}(1)}{\st_{G}(n)}\Big)= \psi_1\Big(\tfrac{\st_{G}(2)}{\st_{G}(n)}\Big).
 \]
 Thus $b^{ip^s-jp^t}[h_0,a]\in \tfrac{\st_{G}(2)}{\st_{G}(n)}$, and so $b^{ip^s-jp^t}\equiv [h_0,a]$ modulo $\tfrac{\st_{G}(2)}{\st_{G}(n)}$. Hence
 \[
 b^{ip^s-jp^t}\in \Big(\frac{G}{\st_{G}(2)}\Big)',
 \]
 which implies, by Lemma~\ref{b-order}, that $ip^s-jp^t\equiv 0 \pmod {p^{m_2}}$, as required. 
  \end{proof}

We now prove that $G/\st_G(n)$ is a Beauville group for all $n\ge 3$. The proof is similar to that of~\cite[Thm.~3.5]{GUA}.

\begin{proof}[Proof of Proposition~\ref{Beauville}(i)]
Certainly $\{a^{-2},ab\}$ and $\{ab^2,b\}$ are both systems of generators of $G/\st_G(n)$. We will show that they yield a Beauville structure for $G/\st_G(n)$.

Set $X=\{a^{-2},ab,a^{-1}b\}$ and $Y=\{ab^2,b,ab^3\}$. We need to prove that
\begin{equation}\label{gen-intersection}
\langle x\rangle \cap \langle y^g\rangle =1
\end{equation}
for all $x\in X$, $y\in Y$ and $g\in G/\st_G(n)$.

Assume first that $x=a^{-2}$. As $x^{p^{m_1-1}}\not \in \st_{G}(1)/\st_G(n)$ but $b^{p^{m_n-1}}\in \st_{G}(1)/\st_G(n)$, it follows that
\[
\langle x^{p^{m_1-1}} \rangle \cap \langle b^{p^{m_n-1}}\rangle ^g=1,
\]
and hence Equation~\eqref{gen-intersection} holds for $x=a^{-2}$ with $y=b$. A similar argument holds for $x=a^{-2}$ with $y=ab^2$ or $y=ab^3$, where here one compares $a^{-2p^{m_1-1}}$ with $(ab^i)^{p^{m_n+m_1-1}}$, for $i\in \{2,3\}$, apart from when $p=3$ where one compares $(ab^3)^{3^{m_n+m_1-2}}$.

For $x=ab$ with $y\in\{ab^2$, $ab^3\}$, the result follows from Lemma~\ref{key}, noting that $p^{m_2}\ge p^2\ge 9$. Likewise for $x=a^{-1}b$ and the above choices for $y$.

It remains to check Equation~\eqref{gen-intersection} for $y=b$ with $x\in\{ab,a^{-1}b\}$. We will check that
\[
\langle (ab)^{p^{m_n+m_1-1}}\rangle \cap \langle b^{p^{m_n-1}}\rangle ^g=1,
\]
and the case $x=a^{-1}b$ follows similarly.

Note that 
\[
\psi_1\big((ab)^{p^{m_n+m_1-1}}\big)=\big((b_1^{\,p^{m_n-1}})^{a_1^*},\ldots ,(b_1^{\,p^{m_n-1}})^{a_1^*}, b_1^{\,p^{m_n-1}}, b_1^{\,p^{m_n-1}}\big)
\]
and 
\[
\psi_1(b^{p^{m_n-1}})=(1,\ldots,1,b_1^{\,p^{m_n-1}}).
\]
Hence any element in $\langle b^{p^{m_n-1}}\rangle ^g$ is in $\text{rist}_G(v)\st_G(n)/\st_G(n)$ for some first-level vertex~$v$. The only element in $\langle (ab)^{p^{m_n+m_1-1 }}\rangle$ that has this property is the trivial element.
Hence we are done.
\end{proof}

Lastly, we consider growing  GGS-groups associated to a defining vector of non-zero sum $s\not\equiv 0 \pmod p$. We show below that all level stabiliser quotients of such growing GGS-group are not Beauville groups.

\begin{lemma}
Let $G$ be a growing GGS-group associated to a defining vector of non-zero sum  $s\not\equiv 0 \pmod p$. For each $n\ge 2$, let $t_{n-1}:=m_1+\cdots+m_{n-1}$. Then the centre $Z(G/\st_G(n))$   contains $\langle (ab)^{p^{t_{n-1}}}\rangle=\langle (a^ib^jc)^{p^{t_{n-1}}}\rangle\cong C_{p^{m_{n}}}$,
for $1\le i\le p^{m_1}-1$, $1\le j \le p^{m_2}-1$ with $i,j\not\equiv 0 \pmod p$ and $c\in \big(G/\st_G(n)\big)'$.
\end{lemma}

\begin{proof}
Let $i\in\{1,\ldots, p^{m_1}-1\}$ and $j\in \{1,\ldots, p^{m_2}-1\}$ with $i,j\not\equiv 0 \pmod p$. Observe that
\[
\psi_1\big((a^ib^jc)^{p^{m_1}}\big)=(a_1^{\,js}b_1^{\,j}c_1,\ldots, a_1^{\,js}b_1^{\,j}c_{p^{{m_1}}})
\]
for $c_1,\ldots, c_{p^{m_1}}\in G_1'$.
Hence, for $w\in \prod_{k=1}^{n-2}X_k$, we have
\begin{align*}
\psi_{n-1}^{w}((a^ib^jc)^{p^{t_{n-1}}})&=(a_{n-1}^{\,js}b_{n-1}^{\,j}d_1,\ldots, a_{n-1}^{\,js}b_{n-1}^{\,j}d_{p^{m_{n-1}} })\\
&\equiv (a_{n-1}^{\,js},\overset{p^{m_{n-1}}}\ldots, a_{n-1}^{\,js}) \pmod {\st_{G_{n-1}}(1)\times \overset{p^{m_{n-1}}}\ldots \times \st_{G_{n-1}}(1)}
\end{align*}
for some $d_1,\ldots,d_{p^{m_{n-1}}} \in G_{n-1}'$. Therefore the result follows.
\end{proof}

\begin{proof}[Proof of Proposition~\ref{Beauville}(ii)]
The result is clear for $n=1$, so we assume $n\ge 2$.  Let $\{x_1,y_1\}$ and $\{x_2,y_2\}$ be two systems of generators. At least one of $x_1,y_1,x_1y_1$, call it $z_1$, must be in the coset $a^ib^jG'$ for $i,j\not\equiv 0 \pmod p$. Likewise for $x_2,y_2,x_2y_2$, and call it $z_2$. Then by the previous lemma, $\langle z_1^{p^{t_{n-1}}}\rangle = \langle z_2^{p^{t_{n-1}}}\rangle $. Hence we are done.
\end{proof}


\bibliographystyle{alpha}

\end{document}